\documentclass[11pt]{amsart}
\usepackage{amssymb,mathrsfs,graphicx,extpfeil}
\usepackage{epsfig}
\usepackage{indentfirst, latexsym, amssymb, enumerate,amsmath,graphicx}
\usepackage{float}
\usepackage{colortbl}
\usepackage{epsfig,subfigure}
\usepackage{amsmath}
\title[The random Cucker-Smale model with uncertain communication]{Local sensitivity analysis for the Cucker-Smale model with random inputs}

\author[Ha]{Seung-Yeal Ha}
\address[Seung-Yeal Ha]{\newline Department of Mathematical Sciences and Research Institute of Mathematics \newline Seoul National University, Seoul 08826 \newline
and Korea Institute for Advanced Study, Hoegiro 87, Seoul 02455, Korea (Republic of)}
\email{syha@snu.ac.kr}

\author[Jin]{Shi Jin}
\address[Shi Jin]{\newline Department of Mathematical Sciences, \newline University of Wisconsin-Madison, Madison, WI 53706, USA}
\email{jin@math.wisc.edu}

\newtheorem{theorem}{Theorem}[section]
\newtheorem{lemma}{Lemma}[section]
\newtheorem{corollary}{Corollary}[section]

\newtheorem{remark}{Remark}[section]

\newtheorem{definition}{Definition}[section]

\newcommand{\bbr}{\mathbb R}

\newcommand{\bbz}{\mathbb Z}

\newcommand{\bbe}{\mathbb E}

\def\charf {\mbox{{\text 1}\kern-.30em {\text l}}}
\begin{document}
%%%%%%%%%%%%%%%%

\date{\today}


\thanks{\textbf{Acknowledgment.} The work of S.-Y. Ha was supported by National Research Foundation of Korea(NRF-2017R1A2B2001864), and the work of S. Jin was supported by NSF grants DMS-1522184 and DMS-1107291: RNMS KI-Net by NSFC grant No. 91330203 and by the Office of the Vice Chancellor for Research and Graduate Education at the University of Wisconsin.}

\begin{abstract}
We present pathwise flocking dynamics and local sensitivity analysis for the Cucker-Smale(C-S) model with random communications and initial data. For the deterministic communications, it is well known that the C-S model can model emergent local and global flocking dynamics depending on  initial data and integrability of communication function. However, the communication mechanism between agents are not a priori clear and needs to be figured out from observed phenomena and data. Thus, uncertainty in communication is an intrinsic component in the flocking modeling of the C-S model. In this paper, we provide a class of admissible random uncertainties which allows us to perform the local sensitivity analysis for flocking and establish stability to the random C-S model with uncertain communication.
\end{abstract}
\maketitle \centerline{\date}

%\tableofcontents

\section{Introduction}
\setcounter{equation}{0}
Emergent phenomena in complex systems are ubiquitous in our nature, to name a few, flocking of birds, aggregation of bacteria and swarming of fish \cite{C-F-T-V, C-H-L, M-T1, T-T} etc. In this paper, we use the terminology "flocking" to denote some concentration phenomena in velocity where individual particles are organized into an ordered motion using the simple rules and environment information through hidden communication mechanism. Modeling of such collective dynamics has received lots of attention in control theory community due to recent  applications in the unmanned aerial vehicles (e.g., drons), unmanned cars and sensor networks \cite{L-P-L-S, P-L-S-G-P, P-E-G}. Several mathematical models were proposed in \cite{C-S, D-M, M-T2, T-T, V-C-B-C-S} and have been studied in relation to the emergent dynamics. Among them, our main interest in this paper lies on the particle model proposed by Cucker and Smale about a decade ago. In \cite{C-S}, Cucker and Smale proposed a Newton-like second-order model and provided some analytical results for the emergent dyanmics of the C-S model. More precisely, let $x_i \in \bbr^d$ and $v_i$ be the position and velocity of the $i$-th C-S particle with unit mass. Then, the dynamics of $(x_i, v_i)$ is governed by the following second-order model:
\begin{align}
\begin{aligned} \label{C-S}
\frac{d}{dt} x_i(t)  &= v_i(t), \quad t > 0, \quad i =1, \cdots, N, \\
\frac{d}{dt} v_i(t) &= \frac{1}{N} \sum_{j=1}^{N} \psi(x_j(t) - x_i(t)) (v_j(t) - v_i(t)),
\end{aligned}
\end{align}
where $\psi = \psi(x)$ is a communication weight function depending on the relative distances between C-S particles. After Cucker-Smale's seminal work in \cite{C-S}, the C-S model \eqref{C-S} has been extensively studied in many aspects, e.g., collision-avoidance \cite{A-C-H-L, C-D2}, stochastic noise effects \cite{A-H, C-M, H-L-L}, emergence of local and global flockings \cite{C-H-H-J-K, C-S, H-Liu, H-T, Sh},  kinetic and hydrodynamic descriptions \cite{C-F-R-T, D-F-T, H-K-K2, K-V}, generalized C-S models \cite{C-D1, K-M-T1, K-M-T2, K-M-T3, M-T2}. The fcuntion $\psi$ is an empirical function, and in real applications, the determination of $\psi$ is based on the phenomenology or modeler's free will. Thus, the uncertainty in $\psi$ is intrinsic. Then, a natural question is how uncertainty in communication  and initial data can affect dynamic features such as flocking and stability. The stochastic noise effect in the C-S flocking was partly addressed in \cite{A-H}, where the communication function $\psi$ is decomposed into a deterministic part and a white noise part. This decomposition turns the deterministic system \eqref{C-S} into a stochastic C-S model with a multiplicative noise. Several elementary stochastic estimates for the velocity process have been studied using Ito's calculus in \cite{A-H, H-L-L}.

In this paper, we consider a general communciation weight function $\psi$ with uncertainty, and we do not assume any specific form for $\psi$ unlike the ansatz in \cite{A-H} so that our model can cover a bounded noises \cite{V-C-B-C-S} as well. With this motivation in mind, we consider a random C-S model with uncertain communication $\psi = \psi(x,z)$ and uncertain initial data depending on $z$, where $z$ is a random variable defined on sample space $\Omega$ with a  probability density function (pdf) $\pi(z)$.
The random processes $(x_i(t,z), v_i(t,z))$ is governed by the following random C-S model:
\begin{align}
\begin{aligned} \label{main}
\partial_t x_i(t,z)  &= v_i(t,z), \quad t > 0, \quad i =1, \cdots, N, \\
\partial_t v_i(t,z) &= \frac{1}{N} \sum_{j=1}^{N} \psi(x_j(t,z) - x_i(t,z), z) (v_j(t,z) - v_i(t,z)),
\end{aligned}
\end{align}
where the communication function $\psi = \psi(x, z) = {\tilde \psi}(|x|, z)$ is radially symmetric in the first argument,  We also assume that $\psi(\cdot, z)$ satisfies several structural properties such as the positivity, boundedness monotonicity and Lipschitz continuity in the first argument: for $z \in \Omega$,
\begin{align}
\begin{aligned} \label{comm}
& 0 < \psi(x, z) \leq \psi_M < \infty, \quad \psi(-x,z) = \psi(x,z), \quad (x,z) \in \bbr^d \times \Omega, \\
&  ({\tilde \psi}(|x_2|, z) - {\tilde \psi}(|x_1|, z)) (|x_2| - |x_1|) \leq 0, \quad \psi(\cdot, z) \in \mbox{Lip}(\bbr^d; \bbr_+),  \\
\end{aligned}
\end{align}
The purpose of this paper is to study theoretical certain aspects of  the C-S model \eqref{main} - \eqref{comm} under the effect of random communications.
Similar analysis has been done recently for kinetic equations with random
uncertainties from initial data and/or collison kernels,  see \cite{H-J, H-J-review, H-J-X, J-L, J-L1, J-L-M, J-X-Z1, J-X-Z2, J-Z, L-W, S-J} for uncertainty quantification in kinetic and hyperbolic models. Recently, the works \cite{A-P-Z, C-P-Z} addressed uncertainty quantification for the swarming models from a numerical point of view. Thus, our work can be viewed as a theoretical justification of their works (see \cite{Xi, X-K} for related works on uncertainty quantification). \newline

The main results of this paper are as follows. First, we give conditions on
the random communication function such that the random dynamical system \eqref{main} exhibits the same asymptotic flocking dynamics  as the deterministic
C-S model \cite{C-S} along the sample path.
 We also provide average dynamics of position and velocity processes. Second, we provide a local sensitivity analysis for \eqref{main} - \eqref{comm}. In particular, we show that the regularity in random space is propagated along the C-S flow and it is stable with respect to initial data. \newline

The rest of this paper is organized as follows. In Section \ref{sec:2}, we review a theoretical minimum for \eqref{main} - \eqref{comm}. In Section \ref{sec:3}, we provide a pathwise flocking and stability dynamics of the random C-S model. In Section \ref{sec:4}, we provide two local sensivity analysis such as the propagation of regularity and its stability analysis with respect to initial data. Finally Section \ref{sec:5} is devoted to the brief summary of our main results and future works. In Appendix A and Appendix B, we present a proof of  Gronwall type lemma and chain rule for higher derivatives of composition of two functions, which are crucially used in Section \ref{sec:3} and Section \ref{sec:4}.

\bigskip

\noindent {\bf Gallery of Notation}: Let $\pi: \Omega \to \bbr_+ \cup \{0\}$ and  $\varphi = \varphi(z)$ be nonnegative pdf function and scalar-valued random function defined on the sample space $\Omega$, respectively. Then, we define the expected value as
\[ \bbe[\varphi] := \int_\Omega \varphi(z) \pi(z) dz, \]
and a weighted $L^2$-space:
\[ L_\pi^2(\Omega) := \{ y:~\Omega \to \bbr~|~\int_{\Omega} |y(z)|^2 \pi(z) dz < \infty \}, \]
with an inner product and norm:
\[ \langle y_1, y_2 \rangle_{L^2_\pi(\Omega)} := \int_{\Omega} y_1(z) y_2(z) \pi(z) dz, \qquad ||y||_{L^2_\pi(\Omega)} := \Big( \int_{\Omega} |y(z)|^2 \pi(z) dz \Big)^{\frac{1}{2}} = \sqrt{\bbe[|y|^2]}. \]
For $k \in \bbz_+ \cup \{0 \}$, we set
\[ ||y||_{H^k_\pi(\Omega)} := \Big( \sum_{\ell = 0}^{k} ||\partial_z^{\ell} y||^2_{L^2_\pi(\Omega)} \Big)^{\frac{1}{2}}, \quad   ||y||_{H_\pi^0(\Omega)} := ||y||_{L_\pi^2(\Omega)}. \]
Moreover, as long as there is no confusion, we suppress $\pi$ and $\Omega$ dependence in $L^2_\pi(\Omega)$-norm and $H^k_\pi(\Omega)$-norm:
\[ ||y||_{L_z^2} :=||y||_{L^2_\pi(\Omega)}, \quad  ||y||_{H_z^k} :=||y||_{H^k_\pi(\Omega)}. \]
For a vector-valued function $y(z) = (y^1(z), \cdots, y^d(z)) \in \bbr^d$, we set
\[ \|y(z)\| := \Big( \sum_{i=1}^{d} |y^i(z)|^2 \Big)^{\frac{1}{2}}, \qquad  ||y||_{L^2_z} := \Big( \sum_{i=1}^{d} ||y^i||^2_{L^2_z} \Big)^\frac{1}{2}.  \]
We also set
\[
  \label{XV}
  X(t,z) := (x_1(t,z), \cdots, x_N(t,z)), \quad V(t,z) := (v_1(t,z), \cdots, v_N(t,z)).
\]
Finally, we use $f \lesssim g$ to denote that there exists a positive generic constant $C$ such that $f \leq C g$.

\section{Preliminaries} \label{sec:2}
\setcounter{equation}{0}
In this section, we briefly review a theoretical minimum for the random C-S flocking model with uncertain communication. First, we recall the local sensitivity analysis of the random ODE system and review the propagation of velocity moments. \newline

\subsection{The random C-S model} Consider an ensemble consisting of $N$-identical C-S particles with the same mass in $\bbr^d$ under uncertain communications registered by $\psi = \psi(x, z)$. \newline

Let $( x_i(t, z), v_i(t,z))\in \mathbb{R}^{2d}$ be the position-velocity processes of the $i$-th particle. Then, their dynamics is governed by the Cauchy problem for the random C-S model:
\[
 \begin{cases}
\displaystyle  \partial_t x_i(t,z) = v_i(t,z), ~~t > 0, \quad i = 1, \cdots, N, \\
\displaystyle  \partial_t v_i(t,z) = \frac{1}{N} \sum_{j=1}^{N} \psi(x_j(t,z) - x_i(t,z), z) (v_j(t,z) - v_i(t,z)), \\
  (x_i(0,z), v_i(0,z)) = (x^0_{i}(z), v^0_{i}(z)).
 \end{cases}
\]
Next, we present definition of {\it pathwise mono-cluster flocking}
for the C-S ensemble.
\begin{definition} \label{D2.1}
\emph{\cite{C-S, H-T}} A random ensemble ${\mathcal P}:= \{(x_i(t,z), v_i(t,z) \}_{i=1}^{N}$ has a asymptotic {\it pathwise mono-cluster flocking} if the ensemble ${\mathcal P}$ satisfies the following two
conditions: for $z \in \Omega$,
\begin{enumerate}
\item
(Formation of velocity alignment):
\[  \lim_{t \to \infty} \max_{i,j} || v_j(t,z) - v_i(t,z)|| = 0.\]

\vspace{0.1cm}

\item
(Formation of group):
\[ \displaystyle \sup_{0 \leq t < \infty} \max_{i,j} || x_j(t,z) - x_i(t,z)|| < \infty. \]
\end{enumerate}
\end{definition}
In the following lemma, we study the time-evolution of first and second velocity moments to be used in later sections.
\begin{lemma} \label{L2.1}
Let $\{ (x_i(t,z), v_i(t,z)) \}_{i=1}^{N}$ be a solution process to the C-S model \eqref{main} - \eqref{comm}. Then, for any $t > 0$ and $z \in \Omega$,  we have
\begin{eqnarray*}
&& (i)~\partial_t \sum_{i=1}^{N} v_i(t,z) =  0. \cr
&& (ii)~\partial_t \sum_{i=1}^{N} \|v_i(t,z) \|^2 = -  \frac{1}{N} \sum_{i,j=1}^{N} \psi(x_j(t,z) - x_i(t,z), z) \|v_j(t,z) - v_i(t,z)\|^2.
\end{eqnarray*}
\end{lemma}
\begin{proof} (i)~We sum $\eqref{main}_2$ over all $i$, and use exchange transformation $(i,j) \leftrightarrow (j,i)
$ to obtain
\begin{eqnarray*}
\partial_t \sum_{i=1}^{N} v_i(t,z) &=& \frac{1}{N} \sum_{i, j=1}^{N} \psi(x_j(t,z) - x_i(t,z), z) (v_j(t,z) - v_i(t,z)) \cr
&=&  -\frac{1}{N} \sum_{i, j=1}^{N} \psi(x_j(t,z) - x_i(t,z), z) (v_j(t,z) - v_i(t,z)) \cr
&=& 0,
\end{eqnarray*}
where we used $\psi(-x,z) = \psi(x,z)$ in \eqref{comm}.  \newline

\noindent (ii) We take an inner product $\eqref{main}_2$ with $2v_i(t,z)$ and sum it over all $i$  to obtain
\begin{align*}
\begin{aligned}
\partial_t \sum_{i=1}^{N} \|v_i(t,z) \|^2 &= \frac{2}{N} \sum_{i,j=1}^{N} \psi(x_j(t,z) - x_i(t,z), z) v_i(t,z) \cdot (v_j(t,z) - v_i(t,z)) \\
&=  -\frac{2}{N} \sum_{i,j=1}^{N} \psi(x_j(t,z) - x_i(t,z), z) v_j(t,z) \cdot (v_j(t,z) - v_i(t,z)) \\
&=  -  \frac{1}{N} \sum_{i,j=1}^{N} \psi(x_j(t,z) - x_i(t,z), z) \|v_j(t,z) - v_i(t,z)\|^2.
\end{aligned}
\end{align*}
\end{proof}
\begin{remark} \label{R2.1}
We set the first and second velocity moments:
\[ m_1(t,z) :=  \sum_{i=1}^{N} v_i(t,z), \quad m_2(t,z) :=  \sum_{i=1}^{N} \|v_i(t,z) \|^2.       \]
As a direct corollary of Lemma \ref{L2.1}, we have the following assertions:
\begin{enumerate}
\item
The modulus of $v_i(t,z)$ is uniformly bounded:
\[
\|v_i(t,z)\|  \leq \sqrt{m_2(t,z)} \leq \sqrt{m_2(0,z)}, \quad z \in \Omega,~~t \geq 0.
\]
This and Cauchy-Schwarz's inequality yield
\[ \sup_{0 \leq t < \infty} \bbe[\|v_i(t)\|] \leq \sqrt{ \bbe[m_2(0)]}. \]
\item
Uniform boundedness of mean and variance:
 \[ \bbe[m_1(t)] = \bbe[m_1(0)], \quad \bbe[m_2(t)] \leq \bbe[m_2(0)], \quad t \geq 0. \]
\end{enumerate}
\end{remark}

\begin{lemma} \label{L2.2}
Let $\{ (x_i(t,z), v_i(t,z)) \}_{i=1}^{N}$ be a solution process to the system \eqref{main}-\eqref{comm} with zero total momentum:
\[ m_1(0,z) = \sum_{i=1}^{N} v^0_i(z) = 0, \quad z \in \Omega. \]
Then, for any $k \geq 1$, we have
\[ \sum_{i=1}^{N} \partial_z^{k} v_i(t,z)  = 0, \quad t > 0,~~ z \in \Omega. \]
\end{lemma}
\begin{proof} Note that the total momentum is conserved along the dynamics \eqref{main} - \eqref{comm}:
\[
\sum_{i=1}^{N} v_i(t,z)  = \sum_{i=1}^{N} v^0_i(z) = 0, \quad t > 0, \quad z \in \Omega.
\]
Then, we differentiate the above relation $k$-times with respect to $z$ to get
\[ \sum_{i=1}^{N} \partial_z^{k} v_i(t,z)  = 0. \quad t > 0.  \]
\end{proof}

\subsection{Local sensivity analysis} Sensivity analysis deals wtih the effects on the output from the input for mathematical and simulation models \cite{S-R}. Below, we briefly discuss local sensitivity analysis following the presentation in \cite{J-T-Z}. Consider the Cauchy problem with the random ordinary differential equations with random initial data:
\[ \label{B-1}
\begin{cases}
\partial_t X(t,z) = F(t, X(t,z), z), \quad t > 0,~~z \in \Omega, \\
X(0,z) = X^0(z),
\end{cases}
\]
where $X:~\bbr_+ \times \Omega \to \bbr^d$ is a random field, and $z$ is system parameters taking a value in $\Omega \subset \bbr^m$. The sensitivity analysis studies the effect on $X$ with respect to the small perturbation $dz$ of $z$:
\[ x_i(t, z + dz) = x_i(t,z) + \sum_{k=1}^m \frac{\partial x_i}{\partial z_k} dz_k + \sum_{i,j = 1}^{m} \frac{\partial^2 x_i}{\partial z_i \partial z_j} dz_i dz_j + \cdots. \]
Then, we define sensivity matrices consisting of coefficients as follows:
\[ S^1 := \Big( \frac{\partial x_i}{\partial z_k} \Big), \quad S^2:= \Big( \frac{\partial^2 x_i}{\partial z_i \partial z_j}  \Big), \cdots  \]
In the following two sections, we will study pathwise flocking estimates, Sobolev estimates and stability estimate of the above sensivity matrices for the random C-S model \eqref{main} - \eqref{comm}.

\section{Pathwise flocking and stability estimates} \label{sec:3}
\setcounter{equation}{0}
 In this section, we present two pathwise estimates, the asymptotic mono-cluster flocking estimate and uniform $\ell_2$-stability estimate for the random C-S model.

 \subsection{Mono-cluster flocking}
 % \label{sec:3.1}
 In this subsection, we
 present a pathwise flocking estimate in Definition \ref{D2.1} using the Lyapunov functional approach \cite{A-C-H-L, H-Liu}. Although the arguments are similar to the case of deterministic C-S model \eqref{C-S}, for reader's convenience, we briefly sketch the pathwise flocking estimate in the sequel. For $X = (x_1, \cdots, x_N) \in \bbr^{dN}$ and $V = (v_1, \cdots, v_N) \in \bbr^{Nd}$, we set
\[ \label{C-1}
||X(t,z)|| := \Big( \sum_{i=1}^{N} ||x_i(t,z)||^2 \Big)^{\frac{1}{2}}, \quad ||V(t,z)|| := \Big( \sum_{i=1}^{N} ||v_i(t,z)||^2 \Big)^{\frac{1}{2}}.
\]
\begin{lemma} \label{L3.1}
Let $\{ (x_i(t,z), v_i(t,z) ) \}_{i=1}^{N}$ be a global smooth solution to \eqref{main} - \eqref{comm} with zero initial total momentum. Then, the functionals $||X||$ and $||V||$ satisfy the system of dissipative differential inequalities (SDDI):  for a.e. $t \in (0, \infty)$ and $z \in \Omega$,
\begin{equation} \label{C-2}
\begin{cases}
\displaystyle \Big| \partial_t ||X(t,z)|| \Big| \leq  ||V(t,z)||, \\
\displaystyle \partial_t ||V(t,z)|| \leq -\psi (\sqrt{2} ||X(t,z)||, z) ||V(t,z)||.
\end{cases}
 \end{equation}
\end{lemma}
\begin{proof}
For fixed $z \in \Omega$, system \eqref{main} is the deterministic C-S system \eqref{C-S}. Hence, we can employ the same argument in Lemma 3.1 of \cite{A-C-H-L}. Thus, we omit its detailed proof.
\end{proof}

\vspace{0.2cm}

Next, we introduce two Lyapunov type functionals ${\mathcal L}_{\pm}(t,z)$:
\[ \label{C-3}
{\mathcal L}_{\pm}(t,z)  := ||V(t,z)||  \pm \frac{1}{\sqrt{2}} \Psi(\sqrt{2} ||X(t,z)||, z), \quad \Psi(x,z) := \int_0^x \psi(\eta,z) d\eta.
\]
Then, we have the following stability estimates in the following lemma.
\begin{lemma} \label{L3.2}
Let $\{ (x_i(t,z), v_i(t,z) ) \}_{i=1}^{N}$ be a global smooth solution to \eqref{main} - \eqref{comm} with zero initial total momentum. Then, the functionals ${\mathcal L}_{\pm}$ satisfy stability estimates: for $t \in (0, \infty)$ and $z \in \Omega$,
\begin{eqnarray*}
&& (i)~{\mathcal L}_{\pm}(t,z) \leq    {\mathcal L}_{\pm}(0,z).    \cr
&& (ii)~||V(t,z)|| + \frac{1}{\sqrt{2}} \Big| \int_{\sqrt{2}||X^0(z)||}^{\sqrt{2}||X(t,z)||} \psi(\eta, z) d\eta \Big| \leq ||V^0(z)||
.\end{eqnarray*}
\end{lemma}
\begin{proof}
The proof is essentially the same as the deterministic case in Lemma 3.2 of \cite{A-C-H-L}. Hence we omit its proof here.
\end{proof}
As a direct application of Lemma \ref{L3.2}, we obtain the emergence of mono-cluster flocking estimate as follows.
\begin{theorem} \label{T3.1}
\emph{(Pathwise flocking estimate)}
For $z \in \Omega$, suppose that the initial data and $\psi$ satisfy the following relation:
\begin{equation} \label{C-4}
||V^0(z)||  < \frac{1}{\sqrt{2}} \int_{\sqrt{2}||X^0(z)||}^{\infty} \psi(s) ds,  \quad m_1(0,z) = \sum_{i=1}^{N} v_i^0 = 0,~~ z \in \Omega,
\end{equation}
and let $\{ (x_i(t,z), v_i(t,z) ) \}_{i=1}^{N}$ be a global smooth solution to \eqref{main} - \eqref{comm}. Then, there exists a positive random variable $x_M(z)$ such  that
\begin{eqnarray*}
&& (i)~\sup_{0 \leq t < \infty} ||X(t,z)||  \leq x_M(z) < \infty. \cr
&& (ii)~||V(t,z)|| \leq ||V^0(z)|| e^{-\psi(\sqrt{2}x_M(z), z) t}, \quad t \geq 0,~~z \in \Omega,
\end{eqnarray*}
where $x_M(z)$ is defined to be the unique value satisfying the following implicit relation:
\begin{equation} \label{C-5}
||V^0(z)|| =  \frac{1}{\sqrt{2}} \int_{\sqrt{2}||X^0(z)||}^{\sqrt{2}x_M(z)} \psi(s,z) ds.
\end{equation}
\end{theorem}
\begin{proof} The proof can be split into two steps. \newline

\noindent $\bullet$ Step A (uniform bound of $||X(\cdot, z)||$): Note that the positivity of $\psi$ implies that the unique determination of $x_M(z)$ via the relation \eqref{C-5}. For such $x_M(z)$, we claim:
\begin{equation} \label{C-6}
  \sup_{0 \leq t < \infty} ||X(t,z)|| \leq x_M(z).
 \end{equation}
Suppose not, i.e., there exists $t_* \in (0, \infty)$ such that
\[   ||X(t_*, z)|| > x_M(z). \]
On the other hand, it follows from (ii) in Lemma \ref{L3.2} that
\[   \frac{1}{\sqrt{2}} \Big| \int_{\sqrt{2}||X^0(z)||}^{\sqrt{2}||X(t_*,z)||} \psi(\eta, z) d\eta \Big| \leq ||V^0(z)||.    \]
This and the relation \eqref{C-5} imply
\[  ||V^0(z)|| = \frac{1}{\sqrt{2}} \int_{\sqrt{2}||X^0(z)||}^{\sqrt{2}x_M(z)} \psi(s, z) ds <  \frac{1}{\sqrt{2}} \int_{\sqrt{2}||X^0(z)||}^{\sqrt{2}X(t_*, z)} \psi(s,z) ds  \leq ||V^0(z)||, \]
which gives a contradiction. Hence we have the uniform boundedness \eqref{C-6} for $||X(\cdot, z)||$.  \newline

\noindent $\bullet$ Step B (exponential decay of $||V(\cdot, z)||$): We use \eqref{comm}, $\eqref{C-2}_2$ and \eqref{C-6} to obtain
\begin{eqnarray*}
\partial_t ||V(t,z)|| &\leq& -\psi (\sqrt{2} ||X(t,z)||, z) ||V(t,z)|| \cr
&\leq& -\psi (\sqrt{2}x_M(z), z) ||V(t,z)||, \quad \mbox{a.e.,}~t \in (0, \infty).
\end{eqnarray*}
Then, Gronwall's lemma yields the desired exponential decay estimate of $||V(\cdot,z)||$.
\end{proof}
\begin{remark} For a given $\psi$, $x_M(z)$ can be found explicitly or implicitly via the relation (\ref{C-5}). Thus, we can write
\[ x_M(z) = x_M( ||X^0(z), ||V^0(z)||). \]
\end{remark}

\vspace{0.5cm}

As a direct corollary of Theorem \ref{T3.1}, we have estimates for the mean of the modulus of $X(t,z)$ and $V(t,z)$, when $\psi$ has a positive lower bound.

\begin{corollary} \label{C3.1} Suppose that $\psi = \psi(x,z)$, and the strength and the initial data satisfy
\begin{equation} \label{NNE-1}
  \inf_{(x, z) \in \bbr \times \Omega} \psi(x,z) \geq \psi_0 > 0, \quad m_1(0) = 0,
\end{equation}
where $\psi_0$ is a positive constant. Then, for a solution $\{ (x_i(t,z), v_i(t,z) ) \}_{i=1}^{N}$ to \eqref{main} with random initial data
$\{ (x^0_i(z), v^0_i(z) ) \}_{i=1}^{N}$, we have
\[ \bbe[ ||V(t)|| ]\leq  \bbe[ ||V^0||] e^{-\psi_0 t}, \qquad  \bbe [ ||X(t)||] \leq \bbe[x_M], \quad t \geq 0.  \]
\end{corollary}
\begin{proof} Note that the relation $\eqref{NNE-1}$ yields
\[ \int_{\sqrt{2}||X^0(z)||}^{\infty} \psi(s,z) ds \geq  \int_{\sqrt{2}||X^0(z)||}^{\infty} \psi_0 ds = \infty. \]
Then, the relation $\eqref{C-4}_1$ holds trivially for any initial data with zero total momentum. Therefore, we have a pathwise mono-cluster flocking:
\begin{equation} \label{C-7}
 ||X(t,z)||  \leq x_M(z) < \infty, \quad ||V(t,z)|| \leq ||V^0(z)|| e^{-\psi(\sqrt{2}x_M(z), z) t}, \quad t \geq 0.
 \end{equation}

\vspace{0.2cm}

\noindent $\bullet$ (Estimate of $\bbe  ||V(t,z)||$): We use the positive lower bound \eqref{NNE-1} for $\psi$ to get
\[ ||V(t,z)|| \leq ||V^0(z)|| e^{-\psi(\sqrt{2}x_M(z), z) t} \leq  ||V^0(z)|| e^{-\psi_0 t}. \]
We multiply $\pi(z)$ to the above relation and integrate it over $\Omega$ to get
\[ \bbe[||V(t)||]  \leq \bbe[ ||V^0||] e^{-\psi_0 t}. \]

\vspace{0.2cm}

\noindent $\bullet$ (Estimate of $\bbe  ||X(t,z)||$): We multiply $\pi(z)$ to $\eqref{C-7}_1$ and integrate it over $\Omega$ to obtain
\[ \bbe[||X(t)||]  \leq \bbe [x_M]. \]
\end{proof}
\begin{remark} Note that the results of Corollary \ref{C3.1} imply
\[   \bbe[ ||v_i(t)||]\leq  \bbe[ ||V^0||] e^{-\psi_0 t} \quad \mbox{and} \quad \bbe [ ||x_i(t)|| ] \leq \bbe[x_M] \quad 1 \leq i \leq N,~t \geq 0.  \]
\end{remark}

\vspace{0.5cm}

\subsection{Uniform $\ell_2$-stability} \label{sec:3.2} In this subsection, we study the uniform $\ell_2$-stability of the random C-S model with respect to initial data along the sample path. First, we recall definition of the uniform $\ell_2$-stability as follows.
\begin{definition} \label{D3.1}
Let $\{ (x_i(t,z), v_i(t,z)) \}_{i=1}^{N}$ and $\{ ({\tilde x}_i(t,z), {\tilde v}_i(t,z)) \}_{i=1}^{N}$ be two smooth solutions to the random C-S model \eqref{main} - \eqref{comm} satisfying zero total momentum and \eqref{C-4}. The random C-S model \eqref{main} is  pathwise uniformly $\ell_2$-stable with respect to initial data, if there exists a positive random variable $G(z)$ independent of $t$ such that
\begin{align}
\begin{aligned} \label{C-8}
&  \sup_{0 \leq t < \infty} \Big(  ||X(t,z) - {\tilde X}(t,z)|| +  || V(t,z) - {\tilde V}(t,z)|| \Big) \\
& \hspace{2cm}  \leq G(z) \Big( || X^0(z) - {\tilde X}^0(z)|| +  || V^0(z) - {\tilde V}^0(z)|| \Big), \quad z \in \Omega.
\end{aligned}
\end{align}
\end{definition}

\vspace{0.2cm}

Before we present the uniform stability estimate along the sample path, we introduce several handy notation in the sequel. Let $\{ (x_i(t,z), v_i(t,z)) \}_{i=1}^{N}$ and $\{ ({\tilde x}_i(t,z), {\tilde v}_i(t,z)) \}_{i=1}^{N}$ be two global solutions to the C-S model \eqref{main} -\eqref{comm}, respectively. Then, for $i, j = 1, \cdots, N,~(t,z) \in \bbr_+ \times \Omega$, we set
\begin{align*}
\begin{aligned}
 x_{ij}(t,z) &:= x_i(t,z) - x_j(t,z), \quad  {\tilde x}_{ij}(t,z) := {\tilde x}_i(t,z) - {\tilde x}_j(t,z),  \\
 v_{ij}(t,z) &:= v_i(t,z) - v_j(t,z), \quad  {\tilde v}_{ij}(t,z) := {\tilde v}_i(t,z) - {\tilde v}_j(t,z),  \\
 \Delta^i_x(t,z) &:= x_i(t,z) - {\tilde x}_i(t,z), \quad \Delta^i_v(t,z) := v_i(t,z) - {\tilde v}_i(t,z),  \\
 \Delta_x(t,z) &:= X(t,z) - {\tilde X}(t,z), \quad  \Delta_v(t,z) := V(t,z) - {\tilde V}(t,z).
\end{aligned}
\end{align*}
Note that $\Delta^i_x(t,z)$ and $\Delta^i_v(t,z)$ satisfy
\begin{align}
\begin{aligned} \label{C-7b}
&\partial_t  \Delta^i_x(t,z) =  \Delta^i_v(t,z), \quad  t>0, \quad 1 \leq i \leq N, \\
&\partial_t \Delta^i_v(t,z) =\frac{1}{N}\sum_{j=1}^N\psi(x_{ji}(t,z), z)\Big(  \Delta^j_v(t,z) -  \Delta^i_v(t,z)  \Big)\\
& \hspace{1.7cm} +  \frac{1}{N}\sum_{j=1}^N \Big(  \psi(x_{ji}(t,z), z) -  \psi({\tilde x}_{ji}(t,z), z) \Big) {\tilde v}_{ji}(t,z).
\end{aligned}
\end{align}
Next, we derive coupled differential inequalities for scalar functionals  $\|\Delta_x\|$ and $\|\Delta_v \|$.
\begin{lemma} \label{L3.3}
Let $\{ (x_i(t,z), v_i(t,z)) \}_{i=1}^{N}$ and $\{ ({\tilde x}_i(t,z), {\tilde v}_i(t,z)) \}_{i=1}^{N}$ be smooth solutions to the C-S model \eqref{main} - \eqref{comm} satisfying zero total momentum and \eqref{C-4}. Then, we have
\begin{align}
\begin{aligned} \label{C-8}
& \Big | \partial_t \| \Delta_x(t,z) \| \Big| \leq \| \Delta_v(t,z) \|, \quad \mbox{a.e.,}~~t > 0, \quad z \in \Omega, \\
&  \partial_t \| \Delta_v(t,z) \| \leq  -\psi_m(z) \|\Delta_v(t,z) \| +  2 \sqrt{2} ||\psi(\cdot, z)||_{Lip}   ||V^0(z)|| \cdot ||\Delta_x(t,z) || e^{-\psi_m(z) t},
\end{aligned}
\end{align}
where the random variable $\psi_m$ is defined by the following relation:
\begin{equation} \label{NNE-2}
\psi_m(z) := \min\{ \psi(\sqrt{2}{\tilde x}_M(z), z),~\psi(\sqrt{2}x_M(z), z) \}.
\end{equation}
\end{lemma}
\begin{proof}
(i) We take an inner product  $\eqref{C-7b}_1$ with $2 \Delta^i_x(t,z)$, sum it over all $i$ to get
\[ \Big| \partial_t ||\Delta^i_x(t,z)||^2  \Big|=  2 |\Delta^i_x(t,z) \cdot \Delta^i_v(t,z) | \leq 2 \|\Delta^i_x(t,z) \| \cdot \| \Delta^i_v(t,z) \|. \]
This yields
\begin{align*}
\begin{aligned}
\Big| \partial_t  ||\Delta_x(t,z)||^2  \Big| &\leq  \sum_{i=1}^{N} \Big| \partial_t  ||\Delta^i_x(t,z)||^2  \Big| \\
&\leq 2  \sum_{i=1}^{N} \|\Delta^i_x(t,z) \| \cdot \| \Delta^i_v(t,z) \| \leq 2 ||\Delta_x(t,z)|| \cdot  ||\Delta_v(t,z)||.
\end{aligned}
\end{align*}
This yields the first differential inequality $\eqref{C-8}_1$.

\vspace{0.2cm}

\noindent (ii) Similarly, we have
\begin{align}
\begin{aligned} \label{C-9}
&\partial_t  \sum_{i=1}^{N} \|\Delta^i_v(t,z)\|^2 \\
& \hspace{1cm} =\frac{2}{N}\sum_{i,j=1}^N\psi(x_{ji}(t,z), z) \Delta_v^i(t,z) \cdot \Big(  \Delta^j_v(t,z) -  \Delta^i_v(t,z)  \Big)\\
&  \hspace{1.3cm}+  \frac{2}{N}\sum_{i,j=1}^N \Big(  \psi(x_{ji}(t,z), z) -  \psi({\tilde x}_{ji}(t,z), z) \Big) \Delta_v^i(t,z) \cdot {\tilde v}_{ji}(t,z) \\
&  \hspace{1cm} = -\frac{1}{N}\sum_{i,j=1}^N\psi(x_{ji}(t,z), z) ||\Delta^j_v(t,z) -  \Delta^i_v(t,z)||^2 \\
& \hspace{1.3cm}+  \frac{2}{N}\sum_{i,j=1}^N \Big(  \psi(x_{ji}(t,z), z) -  \psi({\tilde x}_{ji}(t,z), z) \Big) \Delta_v^i (t,z)\cdot {\tilde v}_{ji}(t,z) \\
& \hspace{1cm}=: {\mathcal I}_{11} + {\mathcal I}_{12}.
\end{aligned}
\end{align}

\vspace{0.2cm}

\noindent $\bullet$~Case A (Estimate of ${\mathcal I}_{11}$): We use the upper bound for $x_{ji}(t,z)$:
\[ \sup_{0 \leq t < \infty} ||x_{ji}(t,z)|| \leq \sqrt{2} x_M(z), \]
and zero total momentum to obtain
\begin{equation} \label{C-10}
{\mathcal I}_{11} \leq -2 \psi(\sqrt{2} x_M(z), z) \|\Delta_v(t,z) \|^2.
\end{equation}

\vspace{0.2cm}

\noindent $\bullet$~Case B (Estimate on ${\mathcal I}_{12}$): For each $z \in \Omega$, it follows from the Lipschitz continuity of $\psi$ and Theorem \ref{T3.1} that we have
\begin{align*}
\begin{aligned}
&|\psi(x_{ji}(t,z), z) -  \psi({\tilde x}_{ji}(t,z), z) | \\
& \hspace{1cm} \leq ||\psi(z)||_{Lip} \|x_{ji}(t,z) - {\tilde x}_{ji}(t,z) \| \leq  ||\psi(z)||_{Lip}( \| \Delta_x^i(t,z) \| +  \| \Delta_x^j(t,z) \| ) \\
& \| {\tilde v}_{ji}(t,z)\| \leq \sqrt{2} ||{\tilde V}(t,z)|| \leq \sqrt{2} ||{\tilde V}^0(z)|| e^{-\psi(\sqrt{2}{\tilde x}_M(z), z) t}.
\end{aligned}
\end{align*}
Here $ ||\psi(z)||_{Lip}$ is the Lipschtz constant of $\psi$. This yields
\begin{align}
\begin{aligned}  \label{C-11}
|{\mathcal I}_{12}| &\leq \frac{2\sqrt{2}}{N} ||\psi(z)||_{Lip}   ||{\tilde V}^0(z)|| e^{-\psi(\sqrt{2} {\tilde x}_M(z), z) t} \\
&\quad \times \sum_{i,j =1}^{N} \Big( ||\Delta_x^i(t,z) || \cdot ||\Delta_v^i(t,z)|| + ||\Delta_x^j(t,z)|| \cdot ||\Delta_v^i(t,z)|| \Big) \\
&= 4 \sqrt{2} ||\psi(z)||_{Lip}   ||V^0(z)|| \cdot ||\Delta_x(t,z) || \cdot ||\Delta_v(t,z) ||  e^{-\psi(\sqrt{2} {\tilde x}_M(z), z) t}.
\end{aligned}
\end{align}
In \eqref{C-9}, we combine estimates \eqref{C-10} and \eqref{C-11} to obtain
\begin{align*}
\begin{aligned}
\partial_t  \|\Delta_v(t,z) \|^2   &\leq -2 \psi(\sqrt{2} x_M(z), z) \|\Delta_v(t,z) \|^2  \\
&\quad+  4 \sqrt{2} ||\psi(z)||_{Lip}   ||V^0(z)|| \cdot ||\Delta_x(t,z) || \cdot ||\Delta_v(t,z) ||  e^{-\psi(\sqrt{2}{\tilde x}_M(z), z) t}.
\end{aligned}
\end{align*}
This and the relation \eqref{NNE-2} yield the desired second inequality:
\begin{align*}
\begin{aligned}
\partial_t \|\Delta_v(t,z) \|  \leq -\psi_m(z) \|\Delta_v(t,z) \| +  2 \sqrt{2} ||\psi(z)||_{Lip}   ||V^0(z)|| \cdot ||\Delta_x(t,z) || e^{-\psi_m(z) t}.
\end{aligned}
\end{align*}
\end{proof}

\begin{lemma}\label{L3.4}
\emph{\cite{H-K-Z}}
Suppose that two nonnegative Lipschitz functions $\mathcal{X}$ and $\mathcal{V}$ satisfy the coupled differential inequalities:
\[
\begin{cases}
\displaystyle \Big| \frac{d{\mathcal X}}{dt} \Big| \leq {\mathcal V}, \quad \frac{d{\mathcal V}}{dt} \leq -\alpha {\mathcal V}  + \gamma e^{-\alpha t} {\mathcal X} + f,  \quad \mbox{a.e.}~~t > 0, \\
\displaystyle ({\mathcal X}(0), {\mathcal V}(0)) = ({\mathcal X}^0, {\mathcal V}^0), \quad t = 0,
\end{cases}
\]
where $\alpha$ and $\gamma$ are positive constants, and $f : {\mathbb R}_+ \cup \{0 \} \rightarrow {\mathbb R}$ is a differentiable, nonnegative, nonincreasing function decaying to zero as its argument goes to infinity and it is integrable. Then, $\mathcal{X}$ and $\mathcal{V}$ satisfy the uniform bound and decay estimates: there exists a positive constant $B_\infty(\alpha, \gamma)$ such that
\begin{align*}
\begin{aligned}
& \mathcal{X}(t)\leq   \Big(1 + \frac{2 B_\infty(\alpha, \gamma) }{\alpha} \Big) ( \mathcal{X}^0+ {\mathcal V}^0 + f(0) + ||f||_{L^1}), \quad t \geq 0, \\
& \mathcal{V}(t)\leq B_\infty (\alpha, \gamma)  ({\mathcal X}^0 + {\mathcal V}^0 + f(0) + ||f||_{L^1} )  e^{-\frac{\alpha}{2} t} +  \frac{1}{\alpha} f \Big(\frac{t}{2} \Big),
\end{aligned}
\end{align*}
where $B_\infty(\alpha, \gamma)$ is a positive constant defined by the following relation:
\[  B_\infty(\alpha, \gamma) := \max \Big \{ \frac{\gamma}{\alpha}, 1  \Big \} \Big( 1 + \frac{8 \gamma}{\alpha^2 e^2}  e^{\gamma \int_0^{\infty} s e^{-\alpha s} ds} \Big). \]
\end{lemma}
\begin{proof}
We leave its proof in Appendix A.
\end{proof}

As a direct application of Lemma \ref{L3.3} and Lemma \ref{L3.4}, we obtain the uniform $\ell_2$-stability with respect to initial data pathwise.
\begin{theorem} \label{T3.2}
Let $\{ (x_i(t,z), v_i(t,z)) \}_{i=1}^{N}$ and $\{ ({\tilde x}_i(t,z), {\tilde v}_i(t,z)) \}_{i=1}^{N}$ be smooth solutions to the C-S model \eqref{main} - \eqref{comm} satisfying zero total momentum and \eqref{C-4}. Then, there exists a positive constant $M_0(z)$ only depending on $(||\psi(z)||_{Lip}, ||V^0||, x_M(z), {\tilde x}_M(z))$ such that
\begin{eqnarray*}
&&  \sup_{0 \leq t < \infty} \Big(  ||X(t,z) - {\tilde X}(t,z)|| +  || V(t,z) - {\tilde V}(t,z)|| \Big) \cr
&& \hspace{2cm}  \leq M_0(z) \Big( || X^0(z) - {\tilde X}^0(z)|| +  || V^0(z) - {\tilde V}^0(z)|| \Big), \quad z \in \Omega.
\end{eqnarray*}
\end{theorem}
\begin{proof} We set
\[ {\mathcal X} :=  \| \Delta_x(t,z) \|, \quad {\mathcal V}:= \| \Delta_v(t,z) \|, \quad \alpha :=   \psi_m(z), \quad \gamma := 2 \sqrt{2} ||\psi(z)||_{Lip}   ||V^0(z)||, \quad f = 0. \]
Then, the constant ${\bar B}_\infty$ becomes
\begin{align*}
\begin{aligned}
{\bar B}_\infty &=   \max \Big \{ \frac{2 \sqrt{2} ||\psi(z)||_{Lip}   ||V^0(z)||}{\psi_m(z)}, 1  \Big \}  \\
&\times \Big( 1 + \frac{16 \sqrt{2} ||\psi(z)||_{Lip}   ||V^0(z)||}{\psi^2_m(z) e^2}    e^{2 \sqrt{2} ||\psi(z)||_{Lip}   ||V^0(z)|| \int_0^{\infty} s e^{-\alpha s} ds} \Big),
\end{aligned}
\end{align*}
where $||\cdot||_{Lip}$ is the Lipschitz semi-norm. Then, Lemma \ref{L3.4} yields
\begin{align}
\begin{aligned} \label{C-12}
\| \Delta_v(t,z) \| &\leq  {\bar B}_\infty \Big( \| \Delta^0_x(z) \|  +  \| \Delta^0_v(z) \|  \Big)e^{-\frac{\psi_m(z)}{2} t } \leq  {\bar B}_\infty \Big( \| \Delta^0_x(z) \|  +  \| \Delta^0_v(z) \|  \Big),  \cr
\| \Delta_x(t,z) \| &\leq \Big[  1 + \frac{2{\bar B}_\infty}{ \psi_m(z)} \Big) \Big] \Big( \| \Delta^0_x(z) \| +  \| \Delta^0_v(z) \| \Big).
\end{aligned}
\end{align}
These yield
\[
\| \Delta_x(t,z) \| +  \| \Delta_v(t,z) \|  \leq  \Big[  1 + {\bar B}_\infty \Big(1 + \frac{2}{ \psi_m(z)} \Big) \Big] \Big( \| \Delta^0_x(z) \| +  \| \Delta^0_v(z) \| \Big).
\]
Finally, we set
\[ M_0(||\psi(z)||_{Lip}, ||V^0||, x_M(z), {\tilde x}_M(z)) := 1 + {\bar B}_\infty \Big(1 + \frac{2}{ \psi_m(z)} \Big) \]
to get the desired estimate.
\end{proof}

\section{Local sensitivity analysis} \label{sec:4}
\setcounter{equation}{0}
In this section, we provide a local sensitivity analysis for position and velocity processes. For the simplicity of presentation, we assume that the random space $\Omega$ is  one-dimensional, i.e., $\Omega \subset \bbr$, and we also recall notation:
 we set (for $k \ge 1$),
\begin{eqnarray*}
%&&  X(t,z) := (x_1(t,z), \cdots, x_N(t,z)), \quad V(t,z) := (v_1(t,z), \cdots,% v_N(t,z)), \cr
&&  \partial_z^k X(t,z) := (\partial_z^k x_1(t,z), \cdots, \partial_z^k x_N(t,z)), \quad \partial_z^k V(t,z) := (\partial_z^k v_1(t,z), \cdots, \partial_z^k v_N(t,z)), \cr
&&  ||X(t)||^2_{H^k_z} := \sum_{i=1}^{N} ||x_i(t)||^2_{H^k_z},  \quad   ||V(t)||^2_{H^k_z} := \sum_{i=1}^{N} ||v_i(t)||^2_{H^k_z}.
\end{eqnarray*}
In the following two subsections, we will derive the following local sensitivity estimates:
\begin{itemize}
\item
(Propagation of regularity in random space): $H_z^k$-regularity of initial data is propagated along the random C-S model \eqref{main}. For any $T \in (0, \infty)$,
\[  ||X^0||_{H^m_z} +  ||V^0||_{H^m_z} < \infty \quad \Longrightarrow \quad \sup_{0 \leq t < T} \Big( \|X(t) \|_{H^m_z} +  \|V(t) \|_{H^m_z} \Big) < \infty. \]
\item
(Stability of $H^k_z$-regularity in random space): $H^k_z$-norm of solution is $\ell_2$-stable with respect to initial data: for two solutions $(X, V)$ and $({\tilde X}, {\tilde V})$ with finite $H^k_z$-norms,
\[ ||X(t) - {\tilde X}(t)||_{H^k_z} +  ||V(t) - {\tilde V}(t)||_{H^k_z}  \lesssim   ||X^0 - {\tilde X}^0||_{H^k_z} +  ||V^0 - {\tilde V}^0||_{H^k_z} + \varepsilon, \quad \varepsilon \ll 1.  \]
\end{itemize}

\subsection{Propagation of $H^k_z$-regularity} \label{sec:4.1}
In this subsection, we present a prolongation of $H^k_z$-regularity along the random C-S flow, once the initial data is $H^k_z$-regular.

\subsubsection{Low-order derivative estimate} In this part, we study the propagation of $H_z^1$-regularity in random space. Suppose that
\[ \sum_{i=1}^{N} \partial_z^k v_i(t,z)  = 0, \quad t > 0,~~ z \in \Omega. \]
Then, it follows from Theorem \ref{T3.1} that for $t \geq 0,~~z \in \Omega$, we have
\[ \label{D-1}
||x_i(t,z)|| \leq x_M(z), \quad ||v_i(t,z)|| \leq  ||V^0(z)|| e^{-\psi(\sqrt{2}x_M(z), z) t}.
\]

Note that $(\partial_z x_i(t,z),  \partial_z v_i(t,z))$ satisfies
\begin{align}
\begin{aligned} \label{D-2}
\partial_t \partial_z x_i(t,z) &= \partial_z v_i(t,z), \quad t > 0, \quad i =1, \cdots, N, \\
\partial_t \partial_z v_i(t,z) &= \frac{1}{N} \sum_{j=1}^{N}  \psi(x_j(t,z) - x_i(t,z), z) (\partial_z v_j(t,z) - \partial_z v_i(t,z)) \\
&+  \frac{1}{N} \sum_{j=1}^{N} \Big[ \nabla_x \psi(x_j(t,z) - x_i(t,z), z) \cdot (\partial_z x_j(t,z) - \partial_z x_i(t,z))  \\
&+ \partial_z \psi (x_j(t,z) - x_i(t,z), z) \Big] \times (v_j(t,z) - v_i(t,z)),
\end{aligned}
\end{align}
By the same argument as in  Lemma \ref{L3.1}, we have the following estimates.
\begin{lemma} \label{L4.1} Suppose that the communication weight function $\psi$ satisfies
\begin{equation*} \label{D-2-1}
\sup_{(x, z) \in \bbr \times \Omega}( \|\nabla_x \psi(x,z)\| + |\partial_z \psi(x,z)|)  \leq \varepsilon_{\psi},
\end{equation*}
and let $\{ (x_i(t,z), v_i(t,z)) \}$ be a solution process to system \eqref{main}-\eqref{comm} with zero total momentum:
\[ \sum_{i=1}^{N} \partial_z^{k} v^0_i(z) = 0, \quad z \in \Omega, \quad  k \leq 1.\]
Then, we have
\begin{align*}
\begin{aligned} \label{D-2-2}
 \Big |\partial_t ||\partial_z X(t,z)|| \Big| &\leq ||\partial_z V(t,z) ||,\quad t > 0,~~z \in \Omega. \\
 \partial_t  ||\partial_z V(t,z)|| &\leq  -2 \psi(\sqrt{2} x_M(z), z) ||\partial_z V(t,z)||  \\
 &+  2 \varepsilon_\psi  ||V^0(z)||  \cdot ||\partial_z X(t,z)||  e^{-\psi(\sqrt{2}x_M(z), z) t} +  \varepsilon_\psi  ||V^0(z)|| e^{-\psi(\sqrt{2}x_M(z), z)t}.
\end{aligned}
\end{align*}
\end{lemma}
\begin{proof} (i) We take an inner product $\eqref{D-2}_1$ with $2  \partial_z x_i(t,z)$ and sum it over all $i$ to get
\[ \Big|\partial_t \sum_{i=1}^{N} ||\partial_z x_i(t,z)||^2 \Big| \leq 2 \Big( \sum_{i=1}^{N} ||\partial_z x_i(t,z) ||^2 \Big)^{\frac{1}{2}} \cdot  \Big( \sum_{i=1}^{N} ||\partial_z v_i(t,z) ||^2 \Big)^{\frac{1}{2}}. \]
This yields the desired estimate. \newline

\noindent (ii) Similarly, we have
\begin{align}
\begin{aligned} \label{D-3}
& \partial_t \sum_{i=1}^{N} |\partial_z v_i(t,z)|^2 \\
& \hspace{1cm} = -\frac{2}{N} \sum_{i,j = 1}^{N} \psi(x_j(t,z) - x_i(t,z), z) \|\partial_z v_j(t,z) - \partial_z v_i(t,z) \|^2  \\
&  \hspace{1cm} +  \frac{2}{N} \sum_{i,j=1}^{N} \Big[ \nabla_x \psi(x_j(t,z) - x_i(t,z), z) \cdot (\partial_z x_j(t,z) - \partial_z x_i(t,z))  \\
&  \hspace{1cm}  + \partial_z \psi (x_j(t,z) - x_i(t,z), z) \Big]  \partial_z v_i(t,z) \cdot (v_j(t,z) - v_i(t,z)) \\
& \hspace{1cm}  =: {\mathcal I}_{21} + {\mathcal I}_{22}.
\end{aligned}
\end{align}
Below, we estimate the term ${\mathcal I}_{2i}$ separately.

\vspace{0.2cm}

\noindent $\bullet$~(Estimate of ${\mathcal I}_{21}$): We use a uniform bound in Theorem \ref{T3.1}:
\[ \|x_j(t,z) - x_i(t,z)\| \leq \sqrt{2} X(t,z) \leq \sqrt{2} x_M(z) \]
to obtain
\begin{equation} \label{D-4}
{\mathcal I}_{21} \leq  -4 \psi(\sqrt{2}x_M(z), z)\sum_{i = 1}^{N} \|\partial_z v_j(t,z)\|^2 = -4 \psi(\sqrt{2}x_M(z), z) ||\partial_z V||^2.
\end{equation}

\vspace{0.2cm}

\noindent $\bullet$~(Estimate of ${\mathcal I}_{22}$): We use exchange transformation $i \longleftrightarrow j$ to find
\begin{align}
\begin{aligned} \label{D-5}
{\mathcal I}_{22} &=  \frac{2}{N} \sum_{i,j=1}^{N} \Big[ \nabla_x \psi(x_j(t,z) - x_i(t,z), z) \cdot (\partial_z x_j(t,z) - \partial_z x_i(t,z))  \\
&\quad+ \partial_z \psi (x_j(t,z) - x_i(t,z), z) \Big]  \partial_z v_i(t,z) \cdot (v_j(t,z) - v_i(t,z)) \\
&= - \frac{2}{N} \sum_{i,j=1}^{N} \Big[ \nabla_x \psi(x_j(t,z) - x_i(t,z), z) \cdot (\partial_z x_j(t,z) - \partial_z x_i(t,z))  \\
&\quad+ \partial_z \psi (x_j(t,z) - x_i(t,z), z) \Big]  \partial_z v_j(t,z) \cdot (v_j(t,z) - v_i(t,z)) \\
& = - \frac{1}{N} \sum_{i,j=1}^{N} \Big[ \nabla_x \psi(x_j(t,z) - x_i(t,z), z) \cdot (\partial_z x_j(t,z) - \partial_z x_i(t,z)) \Big] \\
& \hspace{1.5cm} \times ( \partial_z v_j(t,z) - \partial_z v_i(t,z)) \cdot (v_j(t,z) - v_i(t,z)) \\
& \quad - \frac{1}{N} \sum_{i,j=1}^{N} \Big[ \partial_z \psi (x_j(t,z) - x_i(t,z), z) \Big]  \\
&  \hspace{1.5cm} \times ( \partial_z v_j(t,z) - \partial_z v_i(t,z)) \cdot (v_j(t,z) - v_i(t,z)) \\
& =: {\mathcal I}_{221} + {\mathcal I}_{222}.
\end{aligned}
\end{align}
Next, we estimate the terms ${\mathcal I}_{22i},~i=1,2$ separately. \newline

\noindent $\diamond$ (Estimate of  ${\mathcal I}_{221}$): We use Theorem \ref{T3.1} to obtain
\begin{align}
\begin{aligned} \label{D-5-1}
\|v_j(t,z) - v_i(t,z) \| &\leq   \|v_j(t,z)\| + \|v_i(t,z)\| \leq  \sqrt{2} ||V(t,z)|| \\
&\leq  \sqrt{2} ||V^0(z)|| e^{-\psi(\sqrt{2}x_M(z), z) t}
 \end{aligned}
 \end{align}
to find
\begin{align}
\begin{aligned} \label{D-6}
|{\mathcal I}_{221}| &\leq \frac{\varepsilon_\psi}{N} \sum_{i,j=1}^{N} \|\partial_z x_j - \partial_z x_i\| \cdot \|\partial_z v_j - \partial_z v_i\| \cdot \|v_j - v_i \| \\
&\leq \frac{\sqrt{2} ||V^0(z)|| \varepsilon_\psi}{N} e^{-\psi(\sqrt{2}x_M(z), z) t} \sum_{i,j=1}^{N} \|\partial_z x_j - \partial_z x_i \| \cdot \|\partial_z v_j - \partial_z v_i\| \\
& \leq \frac{\sqrt{2} ||V^0(z)|| \varepsilon_\psi}{N} e^{-\psi(\sqrt{2}x_M(z), z) t}  \Big( \sum_{i,j =1}^{N} \|\partial_z x_j - \partial_z x_i \|^2 \Big)^{\frac{1}{2}} \Big(   \sum_{i,j=1}^{N}  \|\partial_z v_j - \partial_z v_i \|^2 \Big)^{\frac{1}{2}}.
\end{aligned}
\end{align}
On the other hand, we use the conservation of momentum $\sum_{i=1}^{N} \partial_z v_i = 0$ to obtain
\begin{align}
\begin{aligned} \label{D-7}
& \Big[ \sum_{i,j=1}^{N}  \|\partial_z x_j - \partial_z x_i \|^2 \Big]^{\frac{1}{2}} \leq  \Big[ 2 \sum_{i,j=1}^{N} ( \|\partial_z x_j \|^2 + \|\partial_z x_i \|^2 ) \Big]^{\frac{1}{2}} = \sqrt{4N} ||\partial_z X||, \\
& \Big[ \sum_{i,j=1}^{N}  \|\partial_z v_j - \partial_z v_i\|^2 \Big]^{\frac{1}{2}}
 = \Big[ 2 \sum_{i,j=1}^{N} \|\partial_z v_i \|^2 \Big]^{\frac{1}{2}} = \sqrt{2N} ||\partial_z V ||.
\end{aligned}
\end{align}
We combine \eqref{D-6} and \eqref{D-7} to obtain
\begin{equation} \label{D-8}
|{\mathcal I}_{221}| \leq  4 \varepsilon_\psi  ||V^0(z)||  e^{-\psi(\sqrt{2}x_M(z), z) t}  ||\partial_z X|| \cdot  ||\partial_z V ||.
\end{equation}

\vspace{0.2cm}

\noindent $\diamond$ (Estimate of  ${\mathcal I}_{222}$): Similarly, we use \eqref{D-5-1} and \eqref{D-7} to obtain
\begin{align}
\begin{aligned} \label{D-9}
|{\mathcal I}_{222}| &\leq \frac{\varepsilon_\psi}{N} \sum_{i,j=1}^{N} \|\partial_z v_j(t,z) - \partial_z v_i(t,z)\| \cdot \|v_j(t,z) - v_i(t,z)\| \\
 &\leq  \frac{\varepsilon_\psi}{N}  \Big( \sum_{i,j=1}^{N} \|\partial_z v_j(t,z) - \partial_z v_i(t,z)\|^2 \Big)^{\frac{1}{2}} \cdot  \Big( \sum_{i,j=1}^{N} \|v_j(t,z) - v_i(t,z)\|^2 \Big)^{\frac{1}{2}} \\
 &\leq 2 \varepsilon_\psi ||\partial_z V(t,z)|| \cdot ||V(t,z)||  \\
 &\leq 2 \varepsilon_\psi  ||V^0(z)|| e^{-\psi(\sqrt{2}x_M(z), z)t} ||\partial_z V(t,z)||.
\end{aligned}
\end{align}
In \eqref{D-5}, we combine all estimates \eqref{D-8}, \eqref{D-9} to obtain
\begin{align}
\begin{aligned} \label{D-10}
|{\mathcal I}_{22}| &\leq    4 \varepsilon_\psi  ||V^0(z)||  e^{-\psi(2x_M(z), z) t}  ||\partial_z X(t,z)|| \cdot  ||\partial_z V(t,z) || \\
&+ 2  \varepsilon_\psi  ||V^0(z)||  e^{-\psi(\sqrt{2} x_M(z), z)t} ||\partial_z V(t,z)||.
\end{aligned}
\end{align}
Hence in \eqref{D-3}, we again combine \eqref{D-4} and \eqref{D-10} to obtain
\begin{align*}
\begin{aligned}
 \partial_t ||\partial_z V||^2 &\leq -4 \psi(\sqrt{2} x_M(z), z) ||\partial_z V||^2 \\
 &+  4\varepsilon_\psi  ||V^0(z)||  e^{-\psi(\sqrt{2}x_M(z), z) t}  ||\partial_z X|| \cdot  ||\partial_z V || + 2 \varepsilon_\psi  ||V^0(z)||  e^{-\psi(\sqrt{2}x_M(z), z)t} ||\partial_z V||.
\end{aligned}
\end{align*}
This yields the desired estimate.
\end{proof}

\begin{theorem} \label{T4.1}
Suppose that the communication weight function $\psi$ satisfies
\[
\sup_{(x, z) \in \bbr \times \Omega}( |\nabla_x \psi(x,z)| + |\partial_z \psi(x,z)|)  \leq \varepsilon_{\psi} < \infty,
\]
and let $\{ (x_i(t,z), v_i(t,z)) \}$ be a solution process to system \eqref{main}-\eqref{comm} with zero total momentum:
\[ \sum_{i=1}^{N} \partial_z^{k} v^0_i(z) = 0, \quad z \in \Omega, \quad  k \leq 1, \quad ||\partial_z X^0(z)|| + ||\partial_z V^0(z)|| < \infty. \]
Then, there exists a positive constant $M_1 = M_1(z)$ only depending on $\varepsilon_\psi, ||V^0(z)||$ and $||X^0(z)||$ such that
\begin{eqnarray*}
&& (i)~||\partial_z X(t,z)|| \leq M_1(z) \Big (||\partial_z X^0(z)|| + ||\partial_z V^0(z)|| +  \varepsilon_\psi  \Big), \quad t \geq 0,~~z \in \Omega.  \cr
&& (ii)~||\partial_z V(t,z)|| \leq M_1(z) \Big (||\partial_z X^0(z)|| + ||\partial_z V^0(z)|| + \varepsilon_\psi  \Big) e^{-\frac{1}{2}\psi(\sqrt{2}x_M(z), z)t}.
\end{eqnarray*}
\end{theorem}
\begin{proof} We set
\begin{eqnarray*}
&& {\mathcal X} := ||\partial_z X(t,z)||, \quad {\mathcal V} := ||\partial_z V(t,z)||, \quad \alpha := 2 \psi(x_M(z), z), \cr
&& \gamma := 2  \varepsilon_\psi ||V^0(z)||, \quad f(t,z) := \varepsilon_\psi e^{-\psi(\sqrt{2}x_M(z), z) t}.
\end{eqnarray*}
Then, it follows from Lemma \ref{L3.4} that we have
\[ {\tilde A}_\infty :=  e^{2  \varepsilon_\psi ||V^0(z)|| \int_0^{\infty} s e^{-\alpha s} ds}, \qquad  {\tilde B}_\infty := \max \Big \{ \frac{\varepsilon_\psi ||V^0(z)||}{\psi(\sqrt{2} x_M(z), z),}, 1  \Big \} \Big( 1 + \frac{4  \varepsilon_\psi ||V^0(z)|| {\tilde A}_\infty}{\psi^2(\sqrt{2} x_M(z), z) e^2} \Big). \]
With these constants ${\tilde A}_\infty$ and ${\tilde B}_\infty$, we have
\begin{eqnarray*}
&& ||\partial_z V(t,z)|| \leq  {\tilde B}_\infty (||\partial_z X^0(z)|| + ||\partial_z V^0(z)|| + \varepsilon_\psi ) e^{-\psi(\sqrt{2} x_M(z), z) t}, \cr
&& ||\partial_z X(t,z)|| \leq \Big(1 + \frac{2 {\tilde B}_\infty }{2\psi(x_M(z), z)} \Big)  (||\partial_z X^0(z)|| + ||\partial_z V^0(z)|| + \varepsilon_\psi ).
\end{eqnarray*}
Finally, we set
\[ M_1(z) := \max \Big \{   {\tilde B}_\infty(z), 1 + \frac{2 {\tilde B}_\infty(z) }{2\psi(x_M(z), z)} \Big \} \]
to get the desired estimates. Note that $M_1(z)$ depends only on $\varepsilon_\psi, ||V^0(z)||$ and $||X^0(z)||$.
\end{proof}
As a direct corollary of Theorem \ref{T3.1}, we have the following $H^1_z$-estimates.
\begin{corollary} \label{C4.1}
Suppose that the communication weight function $\psi$ satisfies
\[
\sup_{(x,z) \in \bbr^d \times \Omega} \psi(x,z) \geq \psi_0 > 0, \quad \sup_{(x, z) \in \bbr^d \times \Omega}( |\nabla_x \psi(x,z)| + |\partial_z \psi(x,z)|)  \leq \varepsilon_{\psi},
\]
and let $\{ (x_i(t,z), v_i(t,z)) \}$ be a solution process to the system \eqref{main}-\eqref{comm} with zero total momentum:
\[ \sum_{i=1}^{N} \partial_z^{k} v^0_i(z) = 0, \quad z \in \Omega, \quad  k \leq 1.\]
Then, for  $t \geq 0,~~z \in \Omega$, we have
\begin{eqnarray*}
&& (i)~||\partial_z X(t,z)||_{H^1_z} \leq \Big(  \|M_1\|_{H^1_z} ||X^0||_{H^1_z} +   \|M_1\|_{H^1_z} ||V^0||_{H^1_z} + \varepsilon_\psi   \|M_1\|_{H^1_z} \Big).  \cr
&& (ii)~||\partial_z V(t,z)||_{H^1_z} \leq e^{-\frac{\psi_0}{2} t}  \Big(  \|M_1\|_{H^1_z} ||X^0||_{H^1_z} +   \|M_1\|_{H^1_z} ||V^0||_{H^1_z} + \varepsilon_\psi   \|M_1\|_{H^1_z} \Big).
\end{eqnarray*}
\end{corollary}

\subsubsection{Higher-order derivative estimates} \label{sec:4.2} In this part, we present higher-order $H_z^m$-estimates for the propagation of regularity in random space.  \newline

Note that $(\partial_z^m x_i(t,z), \partial_z^m v_i(t,z))$ satisfy
\begin{align}
\begin{aligned} \label{D-11}
&\partial_t \partial_z^m x_i(t,z) = \partial_z^m v_i(t,z), \quad t > 0, \quad i =1, \cdots, N, \\
& \partial_t \partial_z^m v_i(t,z) \\
& \hspace{0.5cm} = \frac{1}{N} \sum_{j=1}^{N}  \psi(x_j(t,z) - x_i(t,z), z) (\partial_z^m v_j(t,z) - \partial_z^m v_i(t,z)) \\
& \hspace{0.8cm} + \underbrace{ \frac{1}{N} \sum_{j=1}^{N} \sum_{k = 0}^{m-1} {{m} \choose {k}} \partial_z^{m-k} (\psi(x_j(t,z) - x_i(t,z), z)) (\partial_z^k v_j(t,z) - \partial_z^k v_i(t,z))}_{=: {\mathcal K}}.
\end{aligned}
\end{align}
Note that the terms ${\mathcal K}$ contain only lower order terms in $(\partial_z^k v_j(t,z) - \partial_z^k v_i(t,z))$ with $k \leq m-1$. As can be seen in Theorems \ref{T3.1} and  \ref{T4.1}, the lower-order terms $V$ and  $\partial_z V$ decays exponentially. Thus, by the induction and Gronwall's type estimate in Lemma \ref{L3.4}, we can expect the exponential decay of higher-order derivatives $\partial_z^{k} V(t,z)$ with $k \geq 2$. This will be made rigorous in the sequel.
\begin{lemma} \label{L4.2}
For $m \in \bbz_+$, suppose that $\psi$ satisfies
\begin{equation} \label{D-11-1}
  \sup_{(x, z) \in \bbr \times \Omega} |\partial^\alpha_{x,z} \psi| \leq \varepsilon_{\psi}, \quad  1 \leq |\alpha| \leq m,
\end{equation}
and let $\{ (x_i(t,z), v_i(t,z)) \}$ be a solution process to system \eqref{main}.  Then, we have
\[ |\partial_z^m [ \psi(x_j(t,z) - x_i(t,z), z) ] | \lesssim \varepsilon_\psi \Big( |\partial^m_z  (x_j(t,z) - x_i(t,z))| + {\mathcal P}_{i,j, m-1} \Big),
\]
where ${\mathcal P}_{i,j,m-1}$ is a polynomial with degree $m-1$ in $\partial_z^k (x_j(t,z) - x_i(t,z)),~0 \leq k \leq m-1.$
\end{lemma}
\begin{proof} Note that for $m \in \bbz_+$,
\begin{align}
\begin{aligned} \label{D-12}
& \partial_z^m \Big( \psi(x_j(t,z) - x_i(t,z), z) \Big) \\
& \hspace{1cm} = \sum_{n=0}^{m} {{m} \choose {n}} \underbrace{(\partial_z^{m-n} \psi)(x_j(t,z) - x_i(t,z), z)}_{=: {\mathcal K}_{00}} \underbrace{\Big[ \nabla_y^n \psi (x_j(t,y) - x_i(t,y), z) \Big] |_{y=z}}_{=:{\mathcal K}_{1n}} \\
\end{aligned}
\end{align}
Note that the term ${\mathcal K}_{00}$ can be estimated using the assumption \eqref{D-11-1}:
\[  |(\partial_z^{m-n} \psi)(x_j(t,z) - x_i(t,z), z) | \leq \varepsilon_\psi.   \]
On the other hand, the term ${\mathcal K}_{1n}$ needs some care to show that it contains the highest derivative term $\nabla_x^n (x_j(t,z) - x_i(t,z))$  and lower-order terms. For this, we have to expand ${\mathcal K}_{1n}$ using the formula in \cite{H-M-Y} (for reader's convenience, we stated it in Appendix \ref{App-B}):
\begin{align*}
\begin{aligned}
{\mathcal K}_{1n} &= \sum \frac{n!}{k_1! k_2! \cdots k_n !} (\nabla_x^k \psi)(x_j(t,y) - x_i(t,y), z)  \Big( \frac{\partial_y x_j(t,y) - \partial_y x_i(t,y)}{1!} \Big)^{k_1}  \\
&\quad \times \Big(  \frac{\partial^2_y x_j(t,y) - \partial^2_y x_i(t,y)}{2!} \Big)^{k_2} \cdots \Big( \frac{\partial^n_y x_j(t,y) - \partial^n_y x_i(t,y)}{n!} \Big)^{k_n},
\end{aligned}
\end{align*}
where the sum is over all nonnegative integer solutions of the Diophantine equation:
\[
 k_1 + 2 k_2 + \cdots + n k_n = n \quad \mbox{and} \quad  k = k_1 + k_2 + \cdots + k_n.
\]
Note that the system of Diophantine equation has a solution
\[ (k_1, \cdots, k_{n-1}, k_n) = (0, \cdots, 0, 1), \quad k = 1, \]
which yields the highest derivative term:
\begin{align}
\begin{aligned} \label{D-13}
{\mathcal K}_{1n} &=   (\nabla_x \psi)(x_j(t,y) - x_i(t,y), z) (\partial^n_y x_j(t,y) - \partial^n_y x_i(t,y)) \\
& + \mbox{lower-order terms like $(\partial^k_y x_j(t,y) - \partial^k_y x_i(t,y))$ ~with $k < n$}.
\end{aligned}
\end{align}
We now combine \eqref{D-12} and \eqref{D-13} to obtain
\begin{align*}
\begin{aligned}
& \partial_z^m \Big( \psi(x_j(t,z) - x_i(t,z), z) \Big) \\
& \hspace{0.5cm} = \sum_{n=0}^{m} {{m} \choose {n}} (\partial_z^{m-n} \psi)(x_j(t,z) - x_i(t,z), z) {\mathcal K}_{1n}(x_j(t,z) - x_i(t,z), z) \\
& \hspace{0.5cm} =  \Big[ (\nabla_x \psi)(x_j(t,y) - x_i(t,y), z)  \sum_{n=0}^{m} {{m} \choose {n}}  (\partial_z^{m-n} \psi)(x_j(t,z) - x_i(t,z), z) \\
& \hspace{0.8cm} \times (\partial^n_y x_j(t,y) - \partial^n_y x_i(t,y)) \Big] + \mbox{lower-order terms} \\
& \hspace{0.5cm} =  (\nabla_x \psi)(x_j(t,y) - x_i(t,y), z)  (\partial^m_y x_j(t,y) - \partial^m_y x_i(t,y)) + \mbox{lower order terms}.
\end{aligned}
\end{align*}
We denote the lower order terms as a polynomial ${\mathcal P}_{i,j, m-1}$ in $\partial_z^k (x_j - x_i),~0 \leq k \leq m-1$.
This yields the desired estimate.
\end{proof}

\begin{lemma} \label{L4.3} For $m \in \bbz_+$, suppose that the communication weight function $\psi$ satisfies
\[
\sup_{(x, z) \in \bbr^d \times \Omega} |\nabla^{\alpha}_{x, z} \psi(x,z)| \leq \varepsilon_{\psi},
\]
and let $\{ (x_i(t,z), v_i(t,z)) \}$ be a solution process to the system \eqref{main}-\eqref{comm} with zero total momentum:
\[ ||X^0||_{H^m_z} +  ||V^0||_{H^m_z} < \infty, \qquad  \sum_{i=1}^{N}  \partial_z^k v^0_i(z) = 0, \quad z \in \Omega, \quad k \leq m. \]
Then, there exists a positive constant ${\bar C}$ such that
\begin{align}
\begin{aligned} \label{D-14}
& (i)~\partial_t ||\partial^m_z X|| \leq ||\partial^m_z V ||,\quad t > 0,~~z \in \Omega. \\
& (ii)~\partial_t  ||\partial^m_z V|| \leq - \psi(\sqrt{2} x_M(z), z) ||\partial_z^m V||  \\
& \hspace{2cm}\quad + {\tilde C} \varepsilon_\psi  ||V(0,z)|| e^{-\psi(\sqrt{2}x_M(z), z) t}  \Big( ||\partial_z^m X|| + \max_{i,j} ||{\mathcal P}_{i,j, m -1}|| \Big)  \\
& \hspace{2cm} \quad +{\tilde C}  \varepsilon_\psi {\mathcal P}_{m-1}(||\partial_z X||, \cdots, ||\partial_z^{m-1} X||) \Big( \max_{1\leq k \leq m-1} \|\partial_z^k V \| \Big).
\end{aligned}
\end{align}
where $C$ is a positive constant appearing in Lemma \ref{L4.2}.
\end{lemma}
\begin{proof} (i) The first inequality follows from the same argument as in (i) Lemma \ref{L4.1}. \newline

\noindent (ii)~We take an inner product $\eqref{D-11}_2$ with $2\partial_z^m v_i(t,z)$ and sum it over all $i$ to get
\begin{align}
\begin{aligned} \label{D-15}
& \partial_t \sum_{i=1}^{N} \|\partial_z^m v_i(t,z) \|^2 \\
& \hspace{1cm} = \frac{2}{N} \sum_{i,j=1}^{N}  \psi(x_j(t,z) - x_i(t,z), z)\partial_z^m v_i(t,z) \cdot (\partial_z^m v_j(t,z) - \partial_z^m v_i(t,z)) \\
& \hspace{1.3cm} +  \frac{2}{N} \sum_{i,j=1}^{N} \sum_{k = 0}^{m-1} {{m} \choose {k}} \partial_z^{m-k} (\psi(x_j(t,z) - x_i(t,z), z)) \\
& \hspace{1.3cm} \times  \partial_z^m v_i(t,z) \cdot (\partial_z^k v_j(t,z) - \partial_z^k v_i(t,z)) \\
& \hspace{1cm} =: {\mathcal I}_{31} + {\mathcal I}_{32}.
\end{aligned}
\end{align}

Next, we estimate the terms ${\mathcal I}_{3i},~i=1,2$ separately. \newline

\noindent $\bullet$ (Estimate of ${\mathcal I}_{31}$):  By the exchange transformation $i \leftrightarrow j$ and zero total momentum, we have
\begin{align}
\begin{aligned} \label{D-19}
{\mathcal I}_{31} &=  \frac{2}{N} \sum_{i,j=1}^{N}  \psi(x_j(t,z) - x_i(t,z), z)\partial_z^m v_i(t,z) \cdot (\partial_z^m v_j(t,z) - \partial_z^m v_i(t,z))  \\
&=  -\frac{2}{N} \sum_{i,j=1}^{N}  \psi(x_j(t,z) - x_i(t,z), z)\partial_z^m v_j(t,z) \cdot (\partial_z^m v_j(t,z) - \partial_z^m v_i(t,z))  \\
&=   -\frac{1}{N} \sum_{i,j=1}^{N}  \psi(x_j(t,z) - x_i(t,z), z) |\partial_z^m v_j(t,z) - \partial_z^m v_i(t,z))|^2 \\
& \leq -2\psi(\sqrt{2}x_M(z), z) \sum_{i=1}^{N} \|\partial_z^m v_i(t,z)\|^2.
\end{aligned}
\end{align}
\noindent $\bullet$ (Estimate of ${\mathcal I}_{32}$): We use Lemma \ref{L4.3} to obtain
\begin{align}
\begin{aligned} \label{D-20}
|{\mathcal I}_{32}| &\leq  \frac{2C \varepsilon_\psi  }{N} \sum_{i,j=1}^{N} \sum_{k = 0}^{m-1} {{m} \choose {k}} ( |\partial^{m-k}_z  x_j| + |\partial^{m-k}_z  x_i|  + {\mathcal P}_{i,j, m-k -1}) \\
&\quad \times |\partial_z^m v_i(t,z)| \cdot |(\partial_z^k v_j - \partial_z^k v_i| \\
&\leq  \frac{2C  \varepsilon_\psi  }{N}  \sum_{i,j=1}^{N} \Big[ ( |\partial^{m}_z  x_j| + |\partial^{m}_z  x_i|  + {\mathcal P}_{i,j, m -1}) |\partial_z^m v_i| \cdot |v_j -v_i| \Big] \\
&\quad + \frac{2C \varepsilon_\psi  }{N} \sum_{i,j=1}^{N} \sum_{k = 1}^{m-1} {{m} \choose {k}} ( |\partial^{m-k}_z  x_j| + |\partial^{m-k}_z  x_i|  + {\mathcal P}_{i,j, m-k -1}) \\
&\times |\partial_z^m v_i| \cdot |(\partial_z^k v_j - \partial_z^k v_i| \\
&=:{\mathcal I}_{321} + {\mathcal I}_{322}.
\end{aligned}
\end{align}

\noindent $\diamond$ (Estimate of ${\mathcal I}_{321}$): We use the flocking estimate in Theorem \ref{T3.1}:
\[  |v_j(t,z) -v_i(t,z)| \leq \sqrt{2} ||V(t,z)|| \leq \sqrt{2}  ||V(0,z)||_{2, \infty} e^{-\psi(\sqrt{2}x_M(z), z) t}   \]
to obtain
\begin{align}
\begin{aligned} \label{D-21}
{\mathcal I}_{321} &= \frac{2C \varepsilon_\psi  }{N}  \sum_{i,j=1}^{N} \Big[ ( |\partial^{m}_z  x_j| + |\partial^{m}_z  x_i|  + {\mathcal P}_{i,j, m -1}) |\partial_z^m v_i| \cdot |v_j -v_i| \Big]  \\
&\leq  \frac{2\sqrt{2} C \varepsilon_\psi  ||V(0,z)|| }{N}  e^{- \psi(\sqrt{2}x_M(z), z) t}  \sum_{i,j=1}^{N} \Big ( |\partial^{m}_z  x_j| + |\partial^{m}_z  x_i|  + {\mathcal P}_{i,j, m -1} \Big) |\partial_z^m v_i|  \\
&\leq  2\sqrt{2} C \varepsilon_\psi  ||V(0,z)|| e^{- \psi(\sqrt{2}x_M(z), z) t}  \Big(2 ||\partial_z^m X|| + \max_{i,j} ||{\mathcal P}_{i,j, m -1}|| \Big)  \|\partial_z^m V\|.
\end{aligned}
\end{align}
\noindent $\diamond$ (Estimate of ${\mathcal I}_{322}$):  By direct calculation, we have
\begin{align}
\begin{aligned} \label{D-22}
{\mathcal I}_{322} &\leq {\mathcal P}_{m-1}(||\partial_z X||, \cdots, ||\partial_z^{m-1} X||) \frac{2\sqrt{2} C  \varepsilon_\psi  }{N} \\
&\quad \times  \sum_{k = 1}^{m-1} {{m} \choose {k}} \sum_{i,j=1}^{N} |\partial_z^m v_i(t,z)| \cdot( |(\partial_z^k v_j(t,z)| +  |\partial_z^k v_i(t,z)| ) \\
&\leq  4 \sqrt{2} C \varepsilon_\psi   (2^m - 2) {\mathcal P}_{m-1}(||\partial_z X(t,z)||, \cdots, ||\partial_z^{m-1} X(t,z)||)  \\
&\quad \times  \|\partial_z^m V(t,z) \| \Big( \max_{1\leq k \leq m-1} \|\partial_z^k V(t,z) \| \Big),
\end{aligned}
\end{align}
where we use the identity  $\sum_{k = 1}^{m-1} {{m} \choose {k}}= 2^m - 2$.  \newline

Now, we combine \eqref{D-21} and \eqref{D-22} to get
\begin{align}
\begin{aligned}  \label{D-23}
|{\mathcal I}_{32}| &\leq  2\sqrt{2} C \varepsilon_\psi  ||V(0,z)||_{2, \infty} e^{-\psi(\sqrt{2}x_M(z), z) t} \\
&\times \Big( 2||\partial_z^m X(t,z)|| + \max_{i,j} ||{\mathcal P}_{i,j, m -1})(t,z)|| \Big)  \|\partial_z^m V(t,z)\| \\
&+4\sqrt{2} C \varepsilon_\psi   (2^m - 2) {\mathcal P}_{m-1}(||\partial_z X(t,z)||, \cdots, ||\partial_z^{m-1} X(t,z)||)  \\
&\times  \|\partial_z^m V(t,z) \| \Big( \max_{1\leq k \leq m-1} \|\partial_z^k V(t,z) \| \Big).
\end{aligned}
\end{align}
Finally, in \eqref{D-3}, we combine \eqref{D-19} and \eqref{D-23} to obtain the desired estimate.
\end{proof}
Next, we use Lemma \ref{L4.3}  to derive a propagation of regularity in random space.

\begin{theorem} \label{T4.2} For $m \in \bbz_+$, suppose that the communication weight function $\psi$ satisfies
\[
\sup_{(x, z) \in \bbr^d \times \Omega} |\nabla^{\alpha}_{x, z} \psi(x,z)| \leq \varepsilon_{\psi},  \quad |\alpha| \leq m,
\]
and let $\{ (x_i(t,z), v_i(t,z)) \}$ be a solution process to the system \eqref{main}-\eqref{comm} with zero total momentum:
\[ ||X^0(z)|| +  ||V^0(z)|| < \infty,~~z \in \Omega, \quad  \sum_{i=1}^{N}  v^0_i(z) = 0, \quad z \in \Omega.\]
Then, there exists positive constant $D_m(z)$ such that for $t \geq 0$,
\[
||\partial_z^m V(t,z)||  \leq D_m(z)  e^{-\frac{1}{2^m} \psi(\sqrt{2} x_M(z), z)t}, \qquad  ||\partial_z^m X(t,z)||  \leq \frac{2^m D_m(z)}{\psi(\sqrt{2} x_M(z), z)}.
\]
\end{theorem}
\begin{proof} For the proof, we use the method of induction. \newline

\noindent $\bullet$~Step A (Initial step): For $m =1$, it follows from Theorem \ref{T4.1} that we have
\begin{align}
\begin{aligned} \label{D-24}
& (i)~||\partial_z X(t,z)|| \leq M_1(z) \Big (||\partial_z X(0,z)|| + ||\partial_z V(0,z)|| + \sqrt{2N}  \varepsilon_\psi  \Big), \quad t \geq 0,~~z \in \Omega, \\
& (ii)~||\partial_z V(t,z)|| \leq M_1 (z) \Big (||\partial_z X(0,z)|| + ||\partial_z V(0,z)|| + \sqrt{2N} \varepsilon_\psi  \Big) e^{-\psi(\sqrt{2} x_M(z), z)t}.
\end{aligned}
\end{align}
We set
\[  D_1(z) := M(z)      \]

\vspace{0.5cm}

\noindent $\bullet$~Step B (Inductive step): Suppose that the estimates hold for $k \leq m-1$, i.e., there exists $D_l(z)$ with $l \leq m-1$ such that
\begin{equation} \label{D-25}
||\partial_z^l V(t,z)||  \leq D_l(z) e^{-\frac{1}{2^l} \psi(\sqrt{2} x_M(z), z)t}, \quad ||\partial_z^l X(t,z)||  \leq \frac{ 2^l D_l(z)}{\sqrt{2} \psi(x_M(z), z)}.
\end{equation}
We substitute these ansatz to \eqref{D-14} to obtain
\begin{align*}
\begin{aligned}
& (i)~\Big| \partial_t ||\partial^m_z X(t,z)|| \Big| \leq ||\partial^m_z V(t,z) ||,\quad t > 0,~~z \in \Omega. \\
& (ii)~\partial_t  ||\partial^m_z V(t,z)|| \leq - \psi(\sqrt{2} x_M(z), z) ||\partial_z^m V(t,z)|| \\
& \hspace{3.3cm} + C_1(z) e^{-\psi(\sqrt{2} x_M(z), z) t}  ||\partial_z^m X(t,z)|| + C_2(z) e^{-\psi(\sqrt{2} x_M(z), z) t}.
\end{aligned}
\end{align*}
We next apply for Lemma \ref{L3.4} to obtain the desired estimates.
\end{proof}

As a direct corollary of Theorem \ref{T4.2}, we have the following $H^1_z$-estimates.
\begin{corollary} \label{C4.2}
For $m \in \bbz_+$, suppose that the communication weight function $\psi$ satisfies
\[
\sup_{(x,z) \in \bbr^d \times \Omega} \psi(x,z) \geq \psi_0 > 0, \quad \sup_{(x, z) \in \bbr^d \times \Omega} |\nabla^{\alpha}_{x, z} \psi(x,z)| \leq \varepsilon_{\psi},  \quad |\alpha| \leq m,
\]
and let $\{ (x_i(t,z), v_i(t,z)) \}$ be a solution process to the system \eqref{main}-\eqref{comm} with zero total momentum:
\[ ||X(0)||_{H^m_z} +  ||V(0)||_{H^m_z} < \infty,~~, \quad  \sum_{i=1}^{N}  v^0_i(z) = 0, \quad z \in \Omega.\]
Then, we have
\[ ||\partial_z X(t,z)||_{H^m_z} \leq \frac{2^m ||D_m||_{H^m_z}}{\psi_0}, \quad \|\partial_z V(t,z)||_{H^m_z} \leq ||D_m||_{H^m_z}  e^{-\frac{1}{2^m} \psi_0 t}. \]
\end{corollary}
\begin{proof}
It follows from Theorem \ref{T4.2} that we have
\begin{equation} \label{D-9}
||\partial_z^m V(t,z)||  \leq D_m(z)  e^{-\frac{1}{2^m} \psi_0 t}, \quad ||\partial_z^m X(t,z)||  \leq \frac{2^m D_m(z)}{\psi_0}.
\end{equation}
This yields the desired estimate.
\end{proof}
\subsection{Uniform stability}  \label{sec:4.3}
In this subsection, we present the uniform $\ell_2$-stability of system \eqref{main} -\eqref{comm}. Let $\{ (x_i(t,z), v_i(t,z)) \}_{i=1}^{N}$ and $\{ ({\tilde x}_i(t,z), {\tilde v}_i(t,z)) \}_{i=1}^{N}$ be smooth solutions to the C-S model \eqref{main}. Then, it is easy to see from \eqref{C-7b} that the differences $\partial_z^m \Delta^i_x(t,z)$ and $\partial_z^m \Delta^i_v(t,z)$ satisfy
\begin{align}
\begin{aligned} \label{D-26}
&\frac{d}{dt}  \partial_z^m \Delta^i_x(t,z) =  \partial_z^m \Delta^i_v(t,z), \quad  t>0, \quad 1 \leq i \leq N, \\
&\frac{d}{dt}  \partial_z^m \Delta^i_v(t,z) =\frac{1}{N}\sum_{j=1}^N\psi(x_{ji}(t,z), z)\Big( \partial_z^m \Delta^j_v(t,z) - \partial_z^m  \Delta^i_v(t,z)  \Big)\\
& \hspace{2.3cm} + \frac{1}{N} \sum_{j=1}^{N} \sum_{k = 0}^{m-1} {{m} \choose {k}} \partial_z^{m-k} (\psi(x_{ji}(t,z), z))
\Big( \partial_z^k \Delta^j_v(t,z) - \partial_z^k  \Delta^i_v(t,z)  \Big) \\
& \hspace{2.3cm} + \frac{1}{N} \sum_{j=1}^{N} \sum_{k = 0}^{m} {{m} \choose {k}} \partial_z^{m-k} \Big(  \psi(x_{ji}(t,z), z) -  \psi({\tilde x}_{ji}(t,z), z) \Big)  \partial_z^k {\tilde v}_{ji}(t,z).
\end{aligned}
\end{align}

\begin{lemma} \label{L4.4}
Let $\{ (x_i(t,z), v_i(t,z)) \}_{i=1}^{N}$ and $\{ ({\tilde x}_i(t,z), {\tilde v}_i(t,z)) \}_{i=1}^{N}$ be smooth solutions to the C-S model \eqref{main} - \eqref{comm} satisfying zero total momentum and \eqref{C-4}. Then, for $m \geq 1$ there exists a positive constant ${\bar D}(m, \varepsilon_\psi)$ only depending on $m$ and $\varepsilon_\psi$ such that
\begin{align}
\begin{aligned} \label{D-27}
& \Big | \frac{d}{dt} \|  \partial_z^m \Delta_x(t,z) \| \Big| \leq \|  \partial_z^m \Delta_v(t,z) \|, \quad \mbox{a.e.,}~~t > 0, \quad z \in \Omega, \\
&\frac{d}{dt}  \| \partial_z^m  \Delta_v(t,z)\| \leq  - \psi_m(z) \| \partial_z^m \Delta_v(t,z) \|  \\
&\hspace{0.5cm} + {\bar D} {\mathcal P}_{m-1}(||\partial_z X(t,z)||, \cdots, ||\partial_z^{m-1} X(t,z)||) \Big( \max_{1\leq k \leq m-1} \|\partial_z^k \Delta_v(t,z) \| \Big) \\
&\hspace{0.5cm} + {\bar D}  \Big[ {\mathcal P}_{m-1}(||\partial_z X(t,z)||, \cdots, ||\partial_z^{m-1} X(t,z)||) +  ||\partial_z^m X(t,z) || + \max_{i,j}  ||P_{i,j, m-1}|| \Big] \\
& \hspace{0.5cm} \times e^{-\frac{1}{2^m} \psi_m(z) t},
\end{aligned}
\end{align}
where $\psi_m(z)$ is a random variable defined in \eqref{NNE-2}.
\end{lemma}
\begin{proof}
(i) We take an inner product  $\eqref{D-26}_1$ with $2 \Delta^i_x(t,z)$, sum it over all $i$ to get
\[ \Big| \frac{d}{dt}  || \partial_z^m \Delta^i_x(t,z)||^2  \Big|=  2 |  \partial_z^m\Delta^i_x(t,z) \cdot  \partial_z^m \Delta^i_v(t,z) | \leq 2 \| \partial_z^m \Delta^i_x(t,z) \| \cdot \|  \partial_z^m\Delta^i_v(t,z) \|. \]
This again yields the first differential inequality in \eqref{D-27}. \newline

\vspace{0.2cm}

\noindent (ii) Similarly, we have
\begin{align}
\begin{aligned} \label{D-28}
&\frac{d}{dt}  \sum_{i=1}^{N} \| \partial_z^m  \Delta^i_v(t,z)\|^2 \\
&\hspace{0.5cm} =\frac{2}{N}\sum_{i,j=1}^N\psi(x_{ji}(t,z), z)  \partial_z^m \Delta^i_v(t,z) \cdot \Big( \partial_z^m \Delta^j_v(t,z) - \partial_z^m  \Delta^i_v(t,z)  \Big)\\
&\hspace{0.8cm} + \frac{2}{N} \sum_{i,j=1}^{N} \sum_{k = 0}^{m-1} {{m} \choose {k}} \partial_z^{m-k} \psi(x_{ji}(t,z), z)  \partial_z^m \Delta^i_v(t,z) \cdot \Big( \partial_z^k \Delta^j_v(t,z) - \partial_z^k  \Delta^i_v(t,z)  \Big) \\
&\hspace{0.8cm} + \frac{2}{N} \sum_{i,j=1}^{N} \sum_{k = 0}^{m} {{m} \choose {k}} \partial_z^{m-k} \Big(  \psi(x_{ji}(t,z), z) -  \psi({\tilde x}_{ji}(t,z), z) \Big)  \partial_z^m \Delta^i_v(t,z) \cdot  \partial_z^k {\tilde v}_{ji}(t,z) \\
&\hspace{0.5cm}  =: {\mathcal I}_{41} + {\mathcal I}_{42} +  {\mathcal I}_{43}.
\end{aligned}
\end{align}

\vspace{0.2cm}

\noindent $\bullet$~Case A (Estimate of ${\mathcal I}_{41}$): Similar to \eqref{D-19}, we use the upper bound for $x_{ji}(t,z)$:
\[ \sup_{0 \leq t < \infty} ||x_{ji}(t,z)|| \leq \sqrt{2}x_M(z), \]
and zero total momentum to obtain
\begin{equation} \label{D-29}
{\mathcal I}_{41} \leq -2 \psi(\sqrt{2} x_M(z), z) \| \partial_z^m \Delta_v(t,z) \|^2.
\end{equation}

\vspace{0.2cm}

\noindent $\bullet$~Case B (Estimate of ${\mathcal I}_{42}$): First, we rewrite ${\mathcal I}_{42}$ as follows.
\begin{align}
\begin{aligned} \label{D-30}
{\mathcal I}_{42} &=  \frac{2}{N} \sum_{i,j=1}^{N} \partial_z^{m} \psi(x_{ji}(t,z), z)  \partial_z^m \Delta^i_v(t,z) \cdot \Big(\Delta^j_v(t,z) - \Delta^i_v(t,z)  \Big) \\
& \quad +  \frac{2}{N} \sum_{i,j=1}^{N} \sum_{k = 1}^{m-1} {{m} \choose {k}} \partial_z^{m-k} \psi(x_{ji}(t,z), z)  \partial_z^m \Delta^i_v(t,z) \cdot \Big( \partial_z^k \Delta^j_v(t,z) - \partial_z^k  \Delta^i_v(t,z)  \Big) \\
& =: {\mathcal I}_{421} +  {\mathcal I}_{422}.
\end{aligned}
\end{align}
Then, we use the same arguments as in \eqref{D-20} to find
\begin{align}
\begin{aligned} \label{D-31}
|{\mathcal I}_{421}| \leq 2\sqrt{2} C \varepsilon_\psi e^{-\frac{\psi_m(z)}{2} t} \Big( ||\partial_z^m X(t,z) || + \max_{i,j}  ||P_{i,j, m-1}|| \Big)  \| \partial_z^m \Delta_v(t,z) \|.
\end{aligned}
\end{align}
and
\begin{align}
\begin{aligned} \label{D-32}
 {\mathcal I}_{422} &\leq {\mathcal P}_{m-1}(||\partial_z X(t,z)||, \cdots, ||\partial_z^{m-1} X(t,z)||) \frac{2\sqrt{2} C  \varepsilon_\psi  }{N} \\
&\quad \times  \sum_{k = 1}^{m-1} {{m} \choose {k}} \sum_{i,j=1}^{N} | \partial_z^m \Delta^i_v(t,z)| \cdot \Big( |\partial_z^k \Delta^j_v(t,z)| + |\partial_z^k  \Delta^i_v(t,z) | \Big)  \\
&\leq  4\sqrt{2} C \varepsilon_\psi   (2^m - 2) {\mathcal P}_{m-1}(||\partial_z X(t,z)||, \cdots, ||\partial_z^{m-1} X(t,z)||)  \\
&\quad \times  \|\partial_z^m \Delta_v(t,z) \| \Big( \max_{1\leq k \leq m-1} \|\partial_z^k \Delta_v(t,z) \| \Big).
\end{aligned}
\end{align}
In \eqref{D-30}, we combine estimates \eqref{D-31} and \eqref{D-32} to obtain
\begin{align}
\begin{aligned} \label{D-33}
{\mathcal I}_{42} &\leq 2\sqrt{2} C \varepsilon_\psi e^{-\frac{\psi_m(z)}{2} t} \Big( ||\partial_z^m X(t,z) || + \max_{i,j}  ||P_{i,j, m-1}|| \Big)  \| \partial_z^m \Delta_v(t,z) \| \\
&\quad + 4\sqrt{2} C \varepsilon_\psi   (2^m - 2) {\mathcal P}_{m-1}(||\partial_z X(t,z)||, \cdots, ||\partial_z^{m-1} X(t,z)||)  \\
&\quad \times  \|\partial_z^m \Delta_v(t,z) \| \Big( \max_{1\leq k \leq m-1} \|\partial_z^k \Delta_v(t,z) \| \Big).
\end{aligned}
\end{align}

\vspace{0.2cm}

\noindent $\bullet$~Case C (Estimate on ${\mathcal I}_{43}$): We use the same arguments in ${\mathcal I}_{42}$ and Theorem \ref{T4.2} that we have
\begin{align}
\begin{aligned} \label{D-34}
{\mathcal I}_{43} &= \frac{2}{N} \sum_{i,j=1}^{N} \sum_{k = 0}^{m} {{m} \choose {k}} \partial_z^{m-k} \Big(  \psi(x_{ji}(t,z), z) -  \psi({\tilde x}_{ji}(t,z), z) \Big) \\
&\quad \times \partial_z^m \Delta^i_v(t,z) \cdot  \partial_z^k {\tilde v}_{ji}(t,z)  \\
&\leq {\mathcal P}_{m-1}(||\partial_z X(t,z)||, \cdots, ||\partial_z^{m-1} X(t,z)||) \frac{2C  \varepsilon_\psi  }{N} \\
&\quad \times  \sum_{k = 0}^{m} {{m} \choose {k}} \sum_{i,j=1}^{N} | \partial_z^m \Delta^i_v(t,z)| \cdot |\partial_z^k {\tilde v}_{ji}(t,z)| \\
& \leq  C 2^{m+1} \varepsilon_\psi  D_m(z) {\mathcal P}_{m-1}(||\partial_z X(t,z)||, \cdots, ||\partial_z^{m-1} X(t,z)||)  \\
&\quad\times || \partial_z^m \Delta_v(t,z)||  e^{-\frac{1}{2^m} \psi(\sqrt{2} x_M(z), z)t}. \\
\end{aligned}
\end{align}
Finally, in \eqref{D-28}, we combine all estimates \eqref{D-29}, \eqref{D-33} and \eqref{D-34} to obtain
\begin{align*}
\begin{aligned} \label{D-35}
\frac{d}{dt}  \| \partial_z^m  \Delta_v(t,z)\|^2 &\leq  -2 \psi(\sqrt{2} x_M(z), z) \| \partial_z^m \Delta_v(t,z) \|^2 \\
&\quad + 2\sqrt{2} C \varepsilon_\psi e^{-\frac{\psi_m(z)}{2} t} \Big( ||\partial_z^m X(t,z) || + \max_{i,j}  ||P_{i,j, m-1}|| \Big)  \| \partial_z^m \Delta_v(t,z) \| \\
&\quad + 4\sqrt{2} C \varepsilon_\psi   (2^m - 2) {\mathcal P}_{m-1}(||\partial_z X(t,z)||, \cdots, ||\partial_z^{m-1} X(t,z)||)  \\
&\quad \times  \|\partial_z^m \Delta_v(t,z) \| \Big( \max_{1\leq k \leq m-1} \|\partial_z^k \Delta_v(t,z) \| \Big) \\
&\quad + C 2^{m+1} \varepsilon_\psi  D_m(z) {\mathcal P}_{m-1}(||\partial_z X(t,z)||, \cdots, ||\partial_z^{m-1} X(t,z)||)  \\
&\quad \times || \partial_z^m \Delta_v(t,z)||  e^{-\frac{1}{2^m} \psi_m(z) t}.
\end{aligned}
\end{align*}
This again yields the desired second estimate in \eqref{D-27}:
\begin{equation}
\begin{aligned} \label{D-36}
&\frac{d}{dt}  \| \partial_z^m  \Delta_v(t,z) \| \\
& \hspace{0.5cm} \leq  - \psi_m(z) \| \partial_z^m \Delta_v(t,z) \| + \sqrt{2} C \varepsilon_\psi e^{-\frac{\psi_m(z)}{2} t} \Big( ||\partial_z^m X(t,z) || + \max_{i,j}  ||P_{i,j, m-1}|| \Big)  \\
&\hspace{0.5cm} + 2\sqrt{2} C \varepsilon_\psi   (2^m - 2) {\mathcal P}_{m-1}(||\partial_z X(t,z)||, \cdots, ||\partial_z^{m-1} X(t,z)||) \Big( \max_{1\leq k \leq m-1} \|\partial_z^k \Delta_v(t,z) \| \Big) \\
&\hspace{0.5cm} + C 2^{m} \varepsilon_\psi  D_m(z) {\mathcal P}_{m-1}(||\partial_z X(t,z)||, \cdots, ||\partial_z^{m-1} X(t,z)||) e^{-\frac{1}{2^m} \psi_m(z) t}.
\end{aligned}
\end{equation}
\end{proof}
Finally, we use the same inductive arguments as in Theorem \ref{T4.2} to obtain the local sensitivity analysis in the uniform stability estimate as follows.

\begin{theorem} \label{T4.3} For $m \in \bbz_+$, suppose that the communication weight function $\psi$ satisfies
\[
\sup_{(x, z) \in \bbr^d \times \Omega} |\nabla^{\alpha}_{x, z} \psi(x,z)| \leq \varepsilon_{\psi},  \quad |\alpha| \leq m,
\]
and let $\{X(t,z), V(t,z)) \}$ and $\{{\tilde X}(t,z), {\tilde V}(t,z)) \}$ be two solution processes to system \eqref{main}-\eqref{comm} with zero total momenta. Then, there exists positive random variable $E_\ell(z)$ and positive constant $\lambda_\ell$ such that for $t \geq 0$,
\begin{align*}
\begin{aligned}
\sum_{k = 0}^{\ell} \| \partial^k_z  \Delta_v(t,z)\|  &\leq \Big(\sum_{k=0}^{\ell} \| \partial^k_z  \Delta_v(0,z)\| +  E_\ell(z)  \Big) e^{-\frac{\psi_m(z)}{\lambda_\ell} t}, \quad t \geq 0,~~z \in \Omega, \\
 \sum_{k = 0}^{\ell} \| \partial^k_z  \Delta_x(t,z)\|  &\leq \max \Big\{1,  \frac{\lambda_\ell}{\psi_m(z)} \Big \} \Big( \sum_{k = 0}^{\ell} \| \partial^k_z  \Delta_x(0,z)\| + \sum_{k=0}^{\ell} \|\partial^k_z  \Delta_v(0,z)\|   +  E_\ell(z)  \Big).
\end{aligned}
\end{align*}
\end{theorem}
\begin{proof} We use the induction argument as in Theorem \ref{T4.2}. For $\ell = 0$, it follows from Theorem \ref{T3.1}, \eqref{C-12} and \eqref{D-24} that
\begin{align}
\begin{aligned} \label{D-37}
\| \Delta_v(t,z) \| &\leq  {\bar B}_\infty \Big( \| \Delta^0_x(z) \|  +  \| \Delta^0_v(z) \|  \Big)e^{-\frac{\psi_m(z)}{2} t },  \\
\| \Delta_x(t,z) \| &\leq \Big[  1 + \frac{2{\bar B}_\infty}{ \psi_m(z)} \Big) \Big] \Big( \| \Delta^0_x(z) \| +  \| \Delta^0_v(z) \| \Big),  \quad ||X(t,z)|| \leq x_M(z), \\
||\partial_z X(t,z)|| &\leq M_1(z) \Big (||\partial_z X(0,z)|| + ||\partial_z V(0,z)|| + \sqrt{2N}  \varepsilon_\psi  \Big).
\end{aligned}
\end{align}
On the other hand, it follows from \eqref{D-27} and \eqref{D-37} that for $\ell = 1$, we have
\begin{align*}
\begin{aligned} \label{D-38}
& \Big | \frac{d}{dt} \|  \partial_z \Delta_x(t,z) \| \Big| \leq \|  \partial_z \Delta_v(t,z) \|, \quad \mbox{a.e.,}~~t > 0, \quad z \in \Omega, \\
&\frac{d}{dt}  \| \partial_z  \Delta_v(t,z)\| \\
& \hspace{0.5cm} \leq  - \psi_m(z) \| \partial_z \Delta_v(t,z) \| + {\bar D} ||\partial_z X(t,z)|| \cdot \| \Delta_v(t,z) \| \\
& \hspace{0.5cm} + {\bar D} \Big(||X(t,z)|| + ||\partial_z X(t,z)||)  \Big) e^{-\frac{1}{2} \psi_m(z) t} \\
& \hspace{0.5cm} \leq  - \psi_m(z) \| \partial_z \Delta_v(t,z) \|  \\
& \hspace{0.5cm} + \Big[  {\bar D} {\bar B}_\infty \Big( \| \Delta^0_x(z) \|  +  \| \Delta^0_v(z) \|  \Big) M_1(z) \Big (||\partial_z X(0,z)|| + ||\partial_z V(0,z)|| + \sqrt{2N}  \varepsilon_\psi  \Big)  \\
& \hspace{0.5cm} +  x_M(z) + M_1(z) \Big (||\partial_z X(0,z)|| + ||\partial_z V(0,z)|| + \sqrt{2N}  \varepsilon_\psi  \Big) \Big] e^{-\frac{\psi_m(z)}{2} t } \\
& \hspace{0.5cm} =: - \psi_m(z) \| \partial_z \Delta_v(t,z) \|  + D_\infty(z) e^{-\frac{\psi_m(z)}{2} t },
\end{aligned}
\end{align*}
i.e., we have
\begin{equation*} \label{D-39}
\begin{cases}
\displaystyle \Big | \frac{d}{dt} \|  \partial_z \Delta_x(t,z) \| \Big| \leq \|  \partial_z \Delta_v(t,z) \|, \quad t > 0,~~z \in \Omega, \\
\displaystyle \frac{d}{dt}  \| \partial_z  \Delta_v(t,z)\| \leq - \psi_m(z) \| \partial_z \Delta_v(t,z) \|  + D_\infty(z) e^{-\frac{\psi_m(z)}{2} t }.
\end{cases}
\end{equation*}
Now, we apply Lemma \ref{AL1.1} in Appendix A for $\eqref{D-39}_2$ with
\[ \alpha =  - \psi_m(z), \qquad f = D_\infty(z) e^{-\frac{\psi_m(z)}{2} t } \]
to obtain
\begin{equation*} \label{D-40}
\| \partial_z  \Delta_v(t,z)\|  \leq  \Big( \| \partial_z  \Delta_v(0,z)\| +  2\frac{D_\infty(z)}{\psi_m(z)}  \Big) e^{-\frac{\psi_m(z)}{4} t}, \quad t \geq 0. \end{equation*}
This and $\eqref{D-39}_1$ imply
\begin{equation} \label{D-41}
  \| \partial_z  \Delta_x(t,z)\| \leq   \| \partial_z  \Delta_x(0,z)\| + \frac{4}{\psi_m(z)} \Big( \| \partial_z  \Delta_v(0,z)\| +  2\frac{D_\infty(z)}{\psi_m(z)}  \Big).
\end{equation}
We set
\[ E_1(z) := 2\frac{D_\infty(z)}{\psi_m(z)}, \quad \lambda_1 :=  4. \]
to get the desired estimate. Other higher-order estimates can be made inductively using the differential inequalities \eqref{D-27} in Lemma \ref{L4.4}. We omit its details.
\end{proof}

\section{Conclusion}\label{sec:5}
\setcounter{equation}{0}
In this paper, we presented local sensitivity analysis for the Cucker-Smale model with random communications. More precisely, we have presented two results. First, we provided conditions on the random commununications  for the pathwise flocking estimates along the sample path which give rise to deterministic flocking asymptotically, and obtained uniform stability analysis with respect to initial data. Second, we performed a local sensitivity analysis for the random Cucker-Smale model. To the best of our knowledge, this is the first theoretical work on the interplay between flocking dynamics and uncertainty quantification. One can also conduct
such analysis for the mean field kinetic equations for flockings \cite{H-T},
and other related models for collective dynamics, decision making and self-organization in complex systems coming from biology and social sciences \cite{Tad}. This will be the subject of future resarch.

\newpage

\appendix

\section{Proof of Lemma \ref{L3.4}} \label{App-1}
\setcounter{equation}{0}
In this appendix, we present a proof of Lemma \ref{L3.4} which is a slight generalization of  Gronwall's inequality appearing in \cite{H-K-Z}. \newline

We basically repeat the same arguments appearing in Lemma 3.1 of \cite{H-K-Z} using a bootstrapping argument. Note that
\begin{align}
\begin{aligned} \label{App-0}
& \Big| \frac{d{\mathcal X}}{dt} \Big| \leq {\mathcal V}, \quad \frac{d{\mathcal V}}{dt} \leq -\alpha {\mathcal V}  + \gamma e^{-\alpha t} {\mathcal X} + f,  \quad \mbox{a.e.}~~t > 0. \\
& ({\mathcal X}(0), {\mathcal V}(0)) = ({\mathcal X}^0, {\mathcal V}^0), \quad t = 0,
\end{aligned}
\end{align}
where $\alpha$ and $\gamma$ are positive constants and $f$ is a nonincreasing function. \newline

Then, we claim:
\begin{align}
\begin{aligned} \label{App-0-1}
& \mathcal{X}(t)\leq   \Big(1 + \frac{2 B_\infty(\alpha, \gamma) }{\alpha} \Big) ( \mathcal{X}^0+ {\mathcal V}^0 + f(0) + ||f||_{L^1}), \quad t \geq 0, \\
& \mathcal{V}(t)\leq B_\infty (\alpha, \gamma)  ({\mathcal X}^0 + {\mathcal V}^0 + f(0) + ||f||_{L^1} )  e^{-\frac{\alpha}{2} t} +  \frac{1}{\alpha} f \Big(\frac{t}{2} \Big).
\end{aligned}
\end{align}
First, we will show that the uniform bound of ${\mathcal V}$, and then we use this uniform bound to derive the exponential decay of ${\mathcal V}$ in two steps.

\vspace{0.2cm}

\noindent $\bullet$ Step A (Uniform boundedness of ${\mathcal V}$): In this step, we derive
\begin{equation} \label{App-0-1-1}
 {\mathcal V}(t) \leq A_\infty(\gamma) \Big( \mathcal{V}^0+ ||f||_{L^1}  +  \frac{\gamma}{\alpha}\mathcal{X}^0 \Big).
 \end{equation}
For this, we set a maximal function ${\mathcal M}_{\mathcal V}$:
\begin{equation} \label{App-0-2}
 {\mathcal M}_{\mathcal V}(t) := \max_{\tau \in [0, t]} {\mathcal V}(\tau), \quad t > 0.
\end{equation}
Next, we will show that
\begin{equation} \label{App-1}
{\mathcal M}_{\mathcal V}(t) \leq A_0(\gamma) \Big( {\mathcal V}^0 + \frac{\gamma}{\alpha} {\mathcal X}^0 + f(0) \Big) , \quad t \geq 0,
\end{equation}
where $A_\infty(\gamma) =  e^{\gamma \int_0^{\infty}} s e^{-\alpha s} ds$. \newline

It follows from the first differential inequality that
\[ \mathcal{X}(t)\leq \mathcal{X}^0 +\int_0^t \mathcal{V}(s)ds. \]
We substitute this into the second differential inequality to get
\begin{align}
\begin{aligned}\label{App-2-1}
\frac{d\mathcal{V}}{dt}&\leq -\alpha \mathcal{V}+\gamma e^{-\alpha t}\mathcal{X} + f \leq -\alpha \mathcal{V}+\gamma e^{-\alpha t}\Big(\mathcal{X}^0+\int_0^t \mathcal{V}(s)ds\Big) + f \\
&\leq \gamma e^{-\alpha t}\Big(\mathcal{X}^0+\int_0^t \mathcal{V}(s)ds\Big) + f.
\end{aligned}
\end{align}
Then, we integrate \eqref{App-2-1} to obtain
\begin{align*}
\begin{aligned}
\mathcal{V}(\tau) &\leq \mathcal{V}^0 + \int_0^\tau f(s) ds +\gamma \int_0^\tau e^{-\alpha s}\Big(\mathcal{X}^0+\int_0^s \mathcal{V}(u)du\Big)ds  \\
&\leq\mathcal{V}^0  + ||f||_{L^1} +\gamma \int_0^t e^{-\alpha s}\Big(\mathcal{X}^0+\int_0^s \mathcal{V}(u)du\Big)ds, \quad \tau \leq t.
\end{aligned}
\end{align*}
This and \eqref{App-0-2} imply
\begin{align}
\begin{aligned} \label{App-3}
{\mathcal M}_{\mathcal V}(t) &\leq \mathcal{V}^0 + ||f||_{L^1} + \gamma \int_0^t e^{-\alpha s}(\mathcal{X}^0+s {\mathcal M}_{\mathcal V}(s) )ds \\
&\leq \mathcal{V}^0+ ||f||_{L^1} + \frac{\gamma}{\alpha} \mathcal{X}^0 + \gamma \int_0^t s e^{-\alpha s} {\mathcal M}_{\mathcal V}(s) ds.
\end{aligned}
\end{align}
We set
\begin{equation} \label{App-5}
Z(t) := \mathcal{V}^0+ ||f||_{L^1}  +  \frac{\gamma}{\alpha}\mathcal{X}^0 + \gamma \int_0^t s e^{-\alpha s} {\mathcal M}_{\mathcal V}(s) ds.
\end{equation}
Then, it follows from \eqref{App-3} and \eqref{App-5} to have
\begin{equation} \label{App-6}
{\mathcal M}_{\mathcal V}(t) \leq Z(t).
\end{equation}
We differentiate $Z(t)$ using the relation \eqref{App-5} and use \eqref{App-6} to obtain
\[
{\dot Z}(t) = \gamma t e^{-\alpha t} {\mathcal M}_{\mathcal V}(t)  \leq \gamma t e^{-\alpha t} Z(t).
\]
This yields
\begin{align}
\begin{aligned} \label{App-6-1}
Z(t) &\leq Z^0 e^{\gamma \int_0^t s e^{-\alpha s} ds} = \Big( \mathcal{V}^0+ ||f||_{L^1}  +  \frac{\gamma}{\alpha}\mathcal{X}^0 \Big) e^{\gamma \int_0^t s e^{-\alpha s} ds} \\
&\leq A_\infty(\gamma) \Big( \mathcal{V}^0+ ||f||_{L^1}  +  \frac{\gamma}{\alpha}\mathcal{X}^0 \Big),
\end{aligned}
\end{align}
where $Z^0 = Z(0),~~A_\infty(\gamma) :=  e^{\gamma \int_0^{\infty} s e^{-\alpha s} ds}.$  \newline

\noindent Then, \eqref{App-6} and \eqref{App-6-1} yield \eqref{App-0-1-1}:
\[ {\mathcal V}(t) \leq M_{{\mathcal V}(t)} \leq Z(t) \leq A_\infty(\gamma) \Big( \mathcal{V}^0+ ||f||_{L^1}  +  \frac{\gamma}{\alpha}\mathcal{X}^0 \Big). \]

\vspace{0.5cm}

\noindent $\bullet$ Step B (Decay estimate of ${\mathcal V}(t)$): We use  \eqref{App-0}, \eqref{App-1} and
\[ ~\max_{0 \leq t < \infty} t e^{-\frac{\alpha t}{2}} = \frac{2}{\alpha e}, \qquad  \max_{0 \leq t < \infty} t^2 e^{-\frac{\alpha t}{2}} = \frac{16}{\alpha^2 e^2} \]
to obtain
\begin{align}
\begin{aligned} \label{App-7}
\mathcal{V}(t)&\leq \mathcal{V}^0 e^{-\alpha t}  +\gamma e^{-\alpha t}\int_0^t \Big(\mathcal{X}^0+\int_0^s\mathcal{V}(\tau)d\tau\Big) ds + \int_0^t e^{-\alpha(t-s)} f(s) ds \\
& \leq \mathcal{V}^0 e^{-\alpha t} +\gamma e^{-\alpha t}\int_0^t \Big[ \mathcal{X}^0 + A_\infty(\gamma) \Big( {\mathcal V}^0 + \frac{\gamma}{\alpha} {\mathcal X}^0 + f(0) \Big) s \Big] ds  \\
&+  e^{-\frac{\alpha t}{2}}  ||f||_{L^1}  + \frac{1}{\alpha} f \Big(\frac{t}{2} \Big) \\
&\leq \mathcal{V}^0 e^{-\frac{\alpha}{2} t}  +\gamma e^{-\frac{\alpha}{2} t} \Big[ \mathcal{X}^0 t  e^{-\frac{\alpha}{2} t}  + \frac{A_\infty(\gamma)}{2} \Big( {\mathcal V}^0 + \frac{\gamma}{\alpha} {\mathcal X}^0 + f(0) \Big) t^2  e^{-\frac{\alpha}{2} t}  \Big] \\
& +  e^{-\frac{\alpha t}{2}} ||f||_{L^1} + \frac{1}{\alpha} f \Big(\frac{t}{2} \Big) \\
& \leq e^{-\frac{\alpha}{2} t} \Big[ \frac{\gamma}{\alpha} \Big( \frac{2}{e} + \frac{8 \gamma A_\infty(\gamma)}{\alpha^2 e^2} \Big) {\mathcal X}^0 +
\Big( 1 + \frac{8 \gamma A_\infty(\gamma)}{\alpha^2 e^2} \Big) {\mathcal V}^0  \\
&+  \frac{8 \gamma A_\infty(\gamma)}{\alpha^2 e^2}  f(0) +  ||f||_{L^1} \Big] +  \frac{1}{\alpha} f \Big(\frac{t}{2} \Big) \\
& \leq \max \Big \{ \frac{\gamma}{\alpha}, 1  \Big \} \Big( 1 + \frac{8 \gamma A_\infty(\gamma)}{\alpha^2 e^2} \Big) ({\mathcal X}^0 + {\mathcal V}^0 + f(0) +   ||f||_{L^1} )  e^{-\frac{\alpha}{2} t} +  \frac{1}{\alpha} f \Big(\frac{t}{2} \Big) \\
& =: B_\infty (\alpha, \gamma)  ({\mathcal X}^0 + {\mathcal V}^0 + f(0) +   ||f||_{L^1} )  e^{-\frac{\alpha}{2} t} +  \frac{1}{\alpha} f \Big(\frac{t}{2} \Big),
\end{aligned}
\end{align}
where in the second inequality, we used the relation:
\begin{align*}
\begin{aligned}
&\int_0^t e^{-\alpha(t-s)} f(s) ds \\
& \hspace{1cm} =  \int_0^{\frac{t}{2}} e^{-\alpha(t-s)} f(s) ds  + \int_{\frac{t}{2}}^t e^{-\alpha(t-s)} f(s) ds  \leq e^{-\frac{\alpha t}{2}} \int_0^{\frac{t}{2}} f(s) ds + f \Big(\frac{t}{2} \Big) \int_{\frac{t}{2}}^t e^{-\alpha(t-s)}ds  \\
& \hspace{1cm} \leq e^{-\frac{\alpha t}{2}}  \Big[  ||f||_{L^1}  -  \frac{1}{\alpha} f \Big(\frac{t}{2} \Big) \Big ] + \frac{1}{\alpha} f \Big(\frac{t}{2} \Big) \leq  e^{-\frac{\alpha t}{2}} ||f||_{L^1}   + \frac{1}{\alpha} f \Big(\frac{t}{2} \Big).
\end{aligned}
\end{align*}

\vspace{0.5cm}

\noindent $\bullet$ Step C (Uniform bound of ${\mathcal X}(t)$): We use \eqref{App-1} and \eqref{App-7}  to obtain
\begin{align}
\begin{aligned} \label{App-8}
\mathcal{X}(t) &\leq \mathcal{X}^0+\int_0^t\mathcal{V}(s)ds \\
& \leq \mathcal{X}^0 + \frac{2B_\infty (\alpha, \gamma)}{\alpha} ({\mathcal X}^0 + {\mathcal V}^0 + f(0) +   ||f||_{L^1} ) + \frac{2}{\alpha} ||f||_{L^1}.
\end{aligned}
\end{align}
Thus, \eqref{App-7} and \eqref{App-8} imply the desired estimates \eqref{App-0-1}. This completes the proof of Lemma  \ref{L3.4}.

\vspace{0.5cm}

By the similar argument, we also have the following Gronwall's lemma.
\begin{lemma} \label{AL1.1}
\emph{\cite{C-H-H-J-K}}
Let $y: {\mathbb R}_+ \cup \{0 \} \to {\mathbb R}_+ \cup \{0 \}$ be a differentiable function satisfying
\[ y^{\prime} \leq - \alpha y + f, \quad t > 0, \qquad y(0) = y_0, \]
where $\alpha$ is a positive constant and $f : {\mathbb R}_+ \cup \{0 \} \rightarrow {\mathbb R}$ is a continuous function decaying to zero as its argument goes to infinity. Then $y$ satisfies
\[ y(t) \leq \frac{1}{\alpha}  \max_{s \in [t/2, t]} |f(s)|+y_0 e^{-\alpha t}+\frac{\|f\|_{L^{\infty}}}{\alpha}e^{-\frac{\alpha t}{2}}, \quad t \geq 0. \]
\end{lemma}
\begin{proof}
Note that $y$ satisfies
\[ y^{\prime} + \alpha y \leq f. \]
We multiply the above differential inequality by $e^{\alpha t}$ and integrate the resulting relation from $s = 0$ to $s=t$ to find
\begin{eqnarray*}
e^{\alpha t} y - y_0 &\leq& \int_0^t f(\tau) e^{\alpha \tau} d\tau \cr
               &=& \int_0^{\frac{t}{2}} f(\tau) e^{\alpha \tau} d\tau + \int_{\frac{t}{2}}^{t} f(\tau) e^{\alpha \tau} d\tau \cr
               &\leq&  \|f\|_{L^{\infty}} \int_0^{\frac{t}{2}} e^{\alpha \tau} d\tau +  \max_{\tau \in [\frac{t}{2},t]} |f(\tau)| \int_{\frac{t}{2}}^{t} e^{\alpha \tau} d\tau \cr
                &\leq&  \frac{\|f\|_{L^{\infty}}}{\alpha} \Big( e^{\frac{\alpha t}{2}} - 1 \Big) + \frac{1}{\alpha}\max_{\tau \in [\frac{t}{2},t]} |f(\tau)|  \Big( e^{\alpha t} - e^{\frac{\alpha t}{2}} \Big).
\end{eqnarray*}
Hence,
\[
 y(t) \leq \frac{1}{\alpha} \max_{\tau \in [\frac{t}{2},t]} |f(\tau)|  +  \Big( y_0 - \frac{\|f\|_{L^{\infty}}}{\alpha}  \Big) e^{-\alpha t} +  \Big( \frac{\|f\|_{L^{\infty}}}{\alpha} -\frac{1}{\alpha}  \max_{\tau \in [\frac{t}{2},t]} |f(\tau)| \Big) e^{-\frac{\alpha t}{2}}.
 \]
 Therefore, for $t \geq 0$,
 \[ y(t) \leq \frac{1}{\alpha} \max_{\tau \in [\frac{t}{2},t]} |f(\tau)|  +  y_0  e^{-\alpha t} + \frac{\|f\|_{L^{\infty}}}{\alpha} e^{-\frac{\alpha t}{2}}. \]
\end{proof}

\section{Chain rules for higher derivatives} \label{App-B}
\setcounter{equation}{0}
In this appendix, we quote the formula for the chain rules for higher derivatives of composition function from \cite{H-M-Y} for reader's convenience. The proof can be made using the mathematical induction.  We first introduce an index set: for given positive integer $n$,
\[ \Lambda(n)  := \{ (k_1, \cdots, k_n) \in (\bbz_+ \cup \{0\} )^n~:~ k_1 + 2 k_2 + \cdots + n k_n = n \}. \]
Note that $(0, \cdots, 0, 1)$ is an element of $\Lambda(n)$. Then, $n$-th derivative of $f(g(x))$ is given by the following formula:
\[ \frac{d^n}{dx^n} f(g(x)) = \sum_{(k_1, \cdots, k_n) \in \Lambda(n)} \frac{n!}{k_1 ! \cdots k_n !} f^{(k)}(g(x)) \Big(\frac{g^{\prime}(x)}{1!}  \Big)^{k_1}
\Big(\frac{g^{\prime \prime}(x)}{2!}  \Big)^{k_2} \cdots  \Big(\frac{g^{(n)}(x)}{n!}  \Big)^{k_n},        \]
where $k := k_1 + \cdots + k_n$.

\end{document}